\documentclass[reqno]{amsart}
\usepackage{hyperref}

\usepackage{hyperref}
\usepackage{placeins}
\usepackage{fullpage}
\usepackage{amsmath}
\usepackage{amssymb}
\usepackage{subfigure}
\usepackage{epsfig}

\usepackage{caption,color}
\usepackage{amsthm}

\usepackage{bm}

\usepackage{url}

\usepackage{accents}

\usepackage{lipsum}

 \newtheorem{prop}{Proposition}

\makeatletter
\newcommand{\vast}{\bBigg@{4.9}}
\newcommand{\Vast}{\bBigg@{5}}
\makeatother

\begin{document}
\title[\hfilneg EJDE-2016/84\hfil Fractional differential equations]
{N-barrier maximum principle for degenerate elliptic systems and its application}

\author[L.-C. Hung \hfil 
\hfilneg]
{Li-Chang Hung}

\address{Li-Chang Hung \newline
Department of Mathematics, National Taiwan University, Taipei, Republic of Taiwan}
\email{lichang.hung@gmail.com}

\author[H.-F. Liu \hfil 
\hfilneg]
{Hsiao-Feng Liu}

\address{Hsiao-Feng Liu \newline
Department of Chemical Engineering, National Taiwan University, Taipei, Taiwan}
\email{b04504060@ntu.edu.tw}

\author[C.-C. Chen \hfil 
\hfilneg]
{Chiun-Chuan Chen}

\address{Chiun-Chuan Chen \newline
Department of Mathematics, National Taiwan University, Taipei, Taiwan}
\email{chchchen@ntu.edu.tw}

\keywords{Maximum principle; traveling wave solutions;
\hfill\break\indent degenerate elliptic systems; reaction-diffusion equations; Lotka-Volterra}

\begin{abstract}
In this paper, we prove the N-barrier maximum principle, which extends the result in \cite{JDE-16} from linear diffusion equations to nonlinear diffusion equations, for a wide class of degenerate elliptic systems of porous medium type. The N-barrier maximum principle provides a priori upper and lower bounds of the solutions to the above-mentioned degenerate nonlinear diffusion equations including the Shigesada-Kawasaki-Teramoto model as a special case. As an application of the N-barrier maximum principle to a coexistence problem in ecology, we show the nonexistence of waves in a three-species degenerate elliptic systems. 
\end{abstract}

\maketitle
\numberwithin{equation}{section}
\newtheorem{theorem}{Theorem}[section]
\newtheorem{lemma}[theorem]{Lemma}
\newtheorem{remark}[theorem]{Remark}
\newtheorem{corollary}[theorem]{Corollary}
\newtheorem{example}[theorem]{Example}
\allowdisplaybreaks

\section{Introduction and main results}

The main perspective of the paper is to establish the \textit{N-barrier maximum principle} (NBMP, see\cite{JDE-16,NBMP-n-species}) for degenerate elliptic systems. To be more precise, we study 
\begin{equation}\label{eqn: degenerate autonomous system of n species}
d_i\,(u_i^{m})_{xx}+\theta\,(u_i)_{x}+u_i^{l_i}\,f_i(u_1,u_2,\cdots,u_n)=0, \ \ x\in\mathbb{R}, \ \ i=1,2,\cdots,n,
\end{equation}
where $u_i=u_i(x)$, $d_i$, $l_i>0$, $\theta\in\mathbb{R}$, and $f_i(u_1,u_2,\cdots,u_n)\in C^0(\mathbb{R^{+}}\times\mathbb{R^{+}}\times\cdots\times\mathbb{R^{+}})$ for $i=1,2,\cdots,n$. The NBMP for the linear diffusion case $m=1$ has been presented in \cite{JDE-16,NBMP-n-species}. In this sequel we will deal with the nonlinear diffusion case $m>1$ based on the N-barrier method developed in \cite{JDE-16,NBMP-n-species}.


We couple \eqref{eqn: degenerate autonomous system of n species} with the prescribed Dirichlet conditions at $x=\pm\infty$:
\begin{equation}\label{eqn: autonomous system of n species BC}
(u_1,u_2,\cdots,u_n)(-\infty)=\textbf{e}_{-},\quad (u_1,u_2,\cdots,u_n)(\infty)=\textbf{e}_{+},
\end{equation}
where 
\begin{equation}\label{eqn: e- and e+}
\textbf{e}_{-}, \textbf{e}_{+}\in\Big\{ (u_1,u_2,\cdots,u_n) \;\Big|\; u_i^{l_i}\,f_i(u_1,u_2,\cdots,u_n)=0\;  (i=1,2,\cdots,n), u_1,u_2,\cdots,u_n\ge 0\Big\}
\end{equation}
are the equilibria of \eqref{eqn: degenerate autonomous system of n species} which connect the solution $(u_1,u_2,\cdots,u_n)(x)$ at $x=-\infty$ and $x=\infty$. This leads to the boundary value problem of \eqref{eqn: degenerate autonomous system of n species} and \eqref{eqn: autonomous system of n species BC}:
\begin{equation*}
\textbf{(BVP)}
\begin{cases}
\vspace{3mm}
d_i\,(u_i^{m})_{xx}+\theta\,(u_i)_{x}+u_i^{l_i}\,f_i(u_1,u_2,\cdots,u_n)=0, \ \ x\in\mathbb{R}, \ \ i=1,2,\cdots,n, \\
(u_1,u_2,\cdots,u_n)(-\infty)=\textbf{e}_{-},\quad (u_1,u_2,\cdots,u_n)(\infty)=\textbf{e}_{+}.
\end{cases}
\end{equation*}

Throughout, we assume, unless otherwise stated, that the following hypothesis on $f_i(u_1,u_2,\cdots,u_n)$ is satisfied:
\begin{itemize}
\item [$\mathbf{[H]}$]
For $i=1,2,\cdots,n$, there exist $\bar{u}_i>\underaccent\bar{u}_i>0$  such that
\begin{eqnarray*}
f_i(u_1,u_2,\cdots,u_n)\ge 0 &\text{ whenever } (u_1,u_2,\cdots,u_n)\in \underaccent\bar{\mathcal{R}};\\\\
f_i(u_1,u_2,\cdots,u_n)\le 0 &\text{ whenever } (u_1,u_2,\cdots,u_n)\in \bar{\mathcal{R}},
\end{eqnarray*}
where 
\begin{eqnarray*}
\underaccent\bar{\mathcal{R}}&=&\bigg\{ (u_1,u_2,\cdots,u_n)\;\Big|\; \sum_{i=1}^{n}\frac{\displaystyle u_i}{\displaystyle\underaccent\bar{u}_i}\le 1,\; u_1,u_2,\cdots,u_n> 0 \bigg\};\\
\bar{\mathcal{R}}&=&\bigg\{ (u_1,u_2,\cdots,u_n)\;\Big|\; \sum_{i=1}^{n}\frac{\displaystyle u_i}{\displaystyle\bar{u}_i}\ge 1,\; u_1,u_2,\cdots,u_n\ge 0 \bigg\}.
\end{eqnarray*}
\end{itemize}
Also, we denote by $\chi$ the characteristic function: 
\begin{equation}\label{eqn: chi 1 or 0}
\chi
=
\begin{cases}
\vspace{3mm}
0,
\quad \text{if} \quad  \text{\bf e}_{+}=(0,\cdots,0) \quad \text{or} \quad \text{\bf e}_{-}=(0,\cdots,0),\\
1,
\quad \text{otherwise}.
\end{cases}
\end{equation}

The main contribution of the N-barrier maximum principle is that it provides rather generic a priori upper and lower bounds for the linear combination of the components of a vector-valued solution which hold for a wide class of reaction terms and boundary conditions. In particular, the key ingredient in the poof relies on the delicate construction of an appropriate N-barrier which allows us to establish the a priori estimates by contradiction. 


 

\begin{theorem}[\textbf{NBMP for $m=1$, \cite{JDE-16,NBMP-n-species}}]\label{thm: NBMP for m=1}
Assume that $\mathbf{[H]}$ holds. Given any set of $\alpha_i>0$ $(i=1,2,\cdots,n)$, suppose that $(u_1(x),u_2(x),\cdots,u_n(x))$ is a nonnegative $C^2$ solution to \textbf{(BVP)} with $m=1$. Then
\begin{equation}\label{eqn: upper and lower bounds of p_generalized}
\underaccent\bar{\lambda}
\leq \sum_{i=1}^{n} \alpha_i\,u_i(x) \leq
\bar{\lambda}, \ \ x\in\mathbb{R},
\end{equation}
where
\begin{eqnarray}
\bar{\lambda} & = &
\frac
{\Big(
\displaystyle
\max_{1\le i \le n}
\alpha_i\,\bar{u}_i
\Big)
\Big(
\displaystyle
\max_{1\le i \le n} d_i 
\Big)}
{
\displaystyle
\min_{1\le i \le n} d_i
} 
,
\\
\underaccent\bar{\lambda} & = & 
\frac
{\Big(
\displaystyle
\min_{1\le i \le n}
\alpha_i\,\underaccent\bar{u}_i
\Big)
\Big(
\displaystyle
\min_{1\le i \le n} d_i 
\Big)}
{
\displaystyle
\max_{1\le i \le n} d_i
} 
\chi,   
\end{eqnarray}
with $\chi$ given by \eqref{eqn: chi 1 or 0}.

\end{theorem}

\textbf{(BVP)} arises from the study of traveling waves in the Shigesada-Kawasaki-Teramoto (SKT) model
\begin{equation}\nonumber
\textbf{(SKT)}
\begin{cases}
\vspace{3mm} 
u_t=\Delta\big(u\,(d_1+\rho_{11}\,u+\rho_{12}\,v)\big)+u\,(\sigma_1-c_{11}\,u-c_{12}\,v),\ \ &y\in\Omega,\ \ t>0,
\\
v_t=\Delta\big(v\,(d_2+\rho_{21}\,u+\rho_{22}\,v)\big)+v\,(\sigma_2-c_{21}\,u-c_{22}\,v),\ \ &y\in\Omega,\ \ t>0,
\end{cases}
\end{equation}
which was proposed by Shigesada, Kawasaki and Teramoto (\cite{Shigesada-Kawasaki-Teramoto-79-SKT-model}) in 1979 to study the spatial segregation problem for two competing species. Here $u(y,t)$ and $v(y,t)$ stand for the density of the two species $u$ and $v$, respectively, and $\Omega\subseteq\mathbb{R}^n$ is the habitat of the two species. $d_1\,\Delta u$ and $d_2\,\Delta v$ come from the random movements of individual species with diffusion rates $d_1$, $d_2>0$. Meanwhile, the terms $\Delta\big(u\,(\rho_{11}\,u+\rho_{12}\,v)\big)$ and $\Delta\big(v\,(\rho_{21}\,u+\rho_{22}\,v)\big)$ include the \textit{self-diffusion} and \textit{cross-diffusion} due to the directed movements of the individuals toward favorable habitats. The coefficients $\rho_{11}$ and $\rho_{22}$ are referred to as the self-diffusion rates, while $\rho_{12}$ and $\rho_{21}$ are the cross-diffusion rates. In addition, the coefficients $\sigma_i$, $c_{ii}$ $(i=1,2)$, and $c_{ij}$ $(i,j=1,2  \ \text{with} \;i\neq j)$ are the intrinsic growth rates, the intra-specific competition rates, and the inter-specific competition rates, which are all assumed to be positive, respectively.

To tackle the problem as to which species will survive in a competitive system is of importance in ecology. To this end, we consider traveling wave solutions, which are solutions of the form
\begin{equation}\label{eqn: traveling wave (u,v)=(U,V) intro}
(u(y,t),v(y,t))=(u(x),v(x)), \quad x=y-\theta \,t,
\end{equation}
where $x\in\mathbb{R}$ and $\theta\in\mathbb{R}$ is the propagation speed of the traveling wave. Ecologically, the sign of $\theta$ indicates which species is stronger and can survive. Inserting \eqref{eqn: traveling wave (u,v)=(U,V) intro} into $\textbf{(SKT)}$ with $\Omega=\mathbb{R}$ leads to 
\begin{equation}\nonumber
\textbf{(SKT-tw)}
\begin{cases}
\vspace{3mm} 
\big(u\,(d_1+\rho_{11}\,u+\rho_{12}\,v)\big)_{xx}+\theta\,u_x+u\,(\sigma_1-c_{11}\,u-c_{12}\,v)=0,\ \ &x\in\mathbb{R},\\
\big(v\,(d_2+\rho_{21}\,u+\rho_{22}\,v)\big)_{xx}+\theta\,v_x+v\,(\sigma_2-c_{21}\,u-c_{22}\,v)=0,\ \ &x\in\mathbb{R}.\\
\end{cases}
\end{equation}
When the self-diffusion and the cross-diffusion effects are neglected or $\rho_{11}=\rho_{12}=\rho_{21}=\rho_{22}=0$, $\textbf{(SKT)}$ with $\Omega=\mathbb{R}$ and $\textbf{(SKT-tw)}$ reduce respectively to 
\begin{equation}\nonumber
\textbf{(LV)}
\begin{cases}
\vspace{3mm} 
u_t=d_1\,\Delta u+u\,(\sigma_1-c_{11}\,u-c_{12}\,v),\ \ &y\in\mathbb{R},\ \ t>0,\\
v_t=d_2\,\Delta v+v\,(\sigma_2-c_{21}\,u-c_{22}\,v),\ \ &y\in\mathbb{R},\ \ t>0,
\end{cases}
\end{equation}
and 
\begin{equation}\nonumber
\textbf{(LV-tw)}
\begin{cases}
\vspace{3mm} 
d_1\,u_{xx}+\theta\,u_x+u\,(\sigma_1-c_{11}\,u-c_{12}\,v)=0,\ \ &x\in\mathbb{R},\\
d_2\,v_{xx}+\theta\,v_x+v\,(\sigma_2-c_{21}\,u-c_{22}\,v)=0,\ \ &x\in\mathbb{R},
\end{cases}
\end{equation}
where $\textbf{(LV)}$ is the celebrated Lotka-Volterra competition-diffusion system of two species and the NBMP for $\textbf{(LV-tw)}$ has been established by applying Theorem~\ref{thm: NBMP for m=1} for \textbf{(LV-tw)} (\cite{JDE-16}).

We illustrate our motivation for establishing Theorem~\ref{thm: NBMP for m=1} for \textbf{(LV-tw)} as follows. When the habitat of the two competing species $u$ and $v$ is resource-limited, the investigation of the total mass or the total density of the two species $v$ and $v$ is essential. This gives rise to the problem of estimating the total density $u(x)+v(x)$ in \textbf{(LV-tw)}. In addition, another issue which motivates us to study the estimate of $u(x)+v(x)$ is the measurement of the \textit{species evenness index} $\mathcal{J}$ for \textbf{(LV-tw)}. $\mathcal{J}$ is defined via \textit{Shannon's diversity index} $\mathcal{H}$ (\cite{Baczkowski98Generalized-diversity-index,Good53Estimation-population-parameters,Ramezan11Shannon-diversity-index,Simpson49Measurement-diversity}), i.e.
\begin{equation}
\mathcal{J}=\frac{\mathcal{H}}{\ln (s)},
\end{equation}
where
\begin{equation}
\mathcal{H}=-\sum_{i=1}^{s} \iota_i\cdot\ln (\iota_i),
\end{equation}
$s$ is the total number of species, and $\iota_i$ is the proportion of the $i$-th species determined by dividing the number of the $i$-th species species by the total number of all species. The species evenness index $\mathcal{J}$ for \textbf{(LV-tw)} is given by
\begin{equation}\label{eqn: species evenness index J for two species}
\mathcal{J}=-\frac{1}{(\ln 2) (u+v)}\,\left(u\,\ln \Big(\frac{u}{u+v}\Big)+v\,\ln \Big(\frac{v}{u+v}\Big)\right).
\end{equation}
We see $u(x)+v(x)$ is involved in the calculation of $\mathcal{J}$. 

Another problem we are concerned with is the parameter dependence on the estimate of $u(x)+v(x)$. When $d_1=d_2$, upper and lower bounds of $u(x)+v(x)$ are given in \cite{CPAA-16} by an approach based on the elliptic maximum principle. For the case of $d_1\neq d_2$, an affirmative answer to an even more general problem of estimating $\alpha\,u+\beta\,v$, where $\alpha,\beta>0$ are arbitrary constants, is given by means of Theorem~\ref{thm: NBMP for m=1}.

On the there hand, we are led to \eqref{eqn: degenerate autonomous system of n species} with $m=n=2$ and $l_i=1$ $(i=1,2,\cdots,n)$ when $d_1=d_2=\rho_{12}=\rho_{21}=0$ in $\textbf{(SKT-tw)}$. We therefore, address the following problem.

\textbf{Q}: \textit{Under $\mathbf{[H]}$, establish the NBMP for $\textbf{(BVP)}$, i.e. find nontrivial lower and upper bounds (depending on the coefficients in $\textbf{(BVP)}$) of $\displaystyle\sum_{i=1}^{n} \alpha_i\,u_i(x)$, where $\alpha_i>0$ $(i=1,2,\cdots,n)$ are arbitrary positive constants.}

Our main result is that \textbf{(BVP)} enjoys the following N-barrier maximum principle, which gives an affirmative answer to \textbf{Q}. Indeed, we have
   
\begin{theorem}[\textbf{NBMP for \textbf{(BVP)}}]\label{thm: NBMP for n species}
Assume that $\mathbf{[H]}$ holds. Given any set of $\alpha_i>0$ $(i=1,2,\cdots,n)$, suppose that $(u_1(x),u_2(x),\cdots,u_n(x))$ is a nonnegative $C^2$ solution to \textbf{(BVP)} with $m>1$. Then
\begin{equation}\label{eqn: upper and lower bounds of p_generalized}
\underaccent\bar{\lambda}
\leq \sum_{i=1}^{n} \alpha_i\,u_i(x) \leq
\bar{\lambda}, \quad x\in\mathbb{R},
\end{equation}
where
\begin{eqnarray}
\bar{\lambda} & = & 
\sqrt[m]{
\Bigg(\sum_{i=1}^{n}\frac{\alpha_i}{\sqrt[m-1]{d_i}}\Bigg)^{2\,(m-1)} 
\bigg(\max_{1\le i \le n} \frac{d_i}{\alpha_i^{m-1}}\bigg)
\Bigg(\max_{1\le i \le n}\alpha_i\,d_i\,\bar{u}_i^m \Bigg)
}
,
\\
\underaccent\bar{\lambda} & = & 
\sqrt[m]{
\Bigg(\sum_{i=1}^{n}\frac{1}{\sqrt[m-1]{\alpha_i\,d_i\,\underaccent\bar{u}_i^m}}\Bigg)^{1-m}
\Bigg(\sum_{i=1}^{n}\frac{\alpha_i}{\sqrt[m-1]{d_i}}\Bigg)^{1-m} 
\bigg(\min_{1\le i \le n} \frac{\alpha_i^{m-1}}{d_i}\bigg)^{2}
}
\chi,   
\end{eqnarray}
with $\chi$ given by \eqref{eqn: chi 1 or 0}.
\end{theorem}
We note that except the case in which either $ \text{\bf e}_{+}=(0,\cdots,0)$ or $ \text{\bf e}_{-}=(0,\cdots,0)$, the boundary conditions at $\pm\infty$ do not play any role in determining the upper and lower bounds given by Theorem~\ref{thm: NBMP for n species}. For either $ \text{\bf e}_{+}=(0,\cdots,0)$ or $ \text{\bf e}_{-}=(0,\cdots,0)$, we clearly have only the trivial lower bound zero.

To illustrate Theorem~\ref{thm: NBMP for n species}, we present an example. Suppose that $m=n=2$, $l_i=1$ and $f_i(u_1,u_2)=u_i\,(\sigma_i-c_{i1}\,u_1-c_{i2}\,u_2)$ ($i=1,2$). Then \textbf{(BVP)} becomes
\begin{equation}\nonumber
\textbf{(NDC-tw)}
\begin{cases}
\vspace{3mm} 
d_1(u_1^2)_{xx}+\theta\,(u_1)_x+u_1\,(\sigma_1-c_{11}\,u_1-c_{12}\,u_2)=0,\ \ &x\in\mathbb{R},\\
\vspace{3mm} 
d_2(u_2^2)_{xx}+\theta\,(u_2)_x+u_2\,(\sigma_2-c_{21}\,u_1-c_{22}\,u_2)=0,\ \ &x\in\mathbb{R},\\
(u_1,u_2)(-\infty)=\textbf{e}_{-},\ \ (u_1,u_2)(+\infty)= \textbf{e}_{+},
\end{cases}
\end{equation}
where
\begin{equation}\label{eqn: e- and e+ 3 species}
\textbf{e}_{-}, \textbf{e}_{+}\in\bigg\{ (u_1,u_2) \;\Big|\; u_i\,(\sigma_i-c_{i1}\,u_1-c_{i2}\,u_2)=0\;  (i=1,2), u_1,u_2\ge 0\bigg\}.
\end{equation}
The degenerate elliptic system \textbf{(NDC-tw)} arises from the study of traveling waves in \textbf{(SKT)} without the presence of diffusion and cross-diffusion, and $\Omega$ replaced by $\mathbb{R}$, i.e.
\begin{equation}\nonumber
\textbf{(NDC)}
\begin{cases}
\vspace{3mm} 
(u_1)_t=d_1(u_1^2)_{yy}+u_1\,(\sigma_1-c_{11}\,u_1-c_{12}\,u_2),\ \ &y\in\mathbb{R},\ \ t>0,\\
(u_2)_t=d_2(u_2^2)_{yy}+u_2\,(\sigma_2-c_{21}\,u_1-c_{22}\,u_2),\ \ &y\in\mathbb{R},\ \ t>0.
\end{cases}
\end{equation}
The nonlinear diffusion-competition system \textbf{(NDC)} has been studied, for example in \cite{Guedda-Kersner-Klincsik-Logak-14-Exact-fronts-periodic}. Under suitable restrictions on the coefficients, explicit spatially periodic stationary solutions to \textbf{(NDC)} can be found. In addition, for appropriate diffusion coefficients the existence of an explicit, unbounded traveling wave to \textbf{(NDC)} is proved under either strong or weak competition. An immediate consequence of Theorem~\ref{thm: NBMP for n species} is the following NBMP for \textbf{(NDC-tw)}.

\begin{corollary}[\textbf{NBMP for NDC-tw}]\label{cor: NBMP for NDC-tw}
Assume that $(u(x),v(x))$ is a nonnegative $C^2$ solution to \textbf{(NDC-tw)}. For any set of $\alpha_i>0$ $(i=1,2)$, we have
\begin{equation}\label{eqn: upper and lower bounds of p_generalized}
\underaccent\bar{\lambda}
\leq \alpha_1\,u_1(x)+\alpha_2\,u_2(x) \leq
\bar{\lambda}, \quad x\in\mathbb{R},
\end{equation}
where
\begin{eqnarray}
\bar{\lambda} & = &
\bigg(
\frac{\alpha_1}{d_1}+\frac{\alpha_2}{d_2}
\bigg) 
\sqrt{
\max
\bigg(
\frac{d_1}{\alpha_1},\frac{d_2}{\alpha_2}
\bigg)
\max
\bigg(
\alpha_1\,d_1\,\bar{u}_1^2,\alpha_2\,d_2\,\bar{u}_2^2
\bigg)
}
,
\\
\underaccent\bar{\lambda} & = &
d_1\,d_2\,\underaccent\bar{u}_1\,\underaccent\bar{u}_2\,\min \Big(\frac{\alpha_1}{d_1},\frac{\alpha_2}{d_2}\Big)\,\sqrt{\frac{\displaystyle \alpha_1\,\alpha_2}{(\displaystyle\alpha_1\,d_1\,\underaccent\bar{u}_1^2+\alpha_2\,d_2\,\underaccent\bar{u}_2^2)\,(\alpha_1\,d_2+\alpha_2\,d_1)}}
\chi
,   
\end{eqnarray}
with $\chi$ given by \eqref{eqn: chi 1 or 0} and
\begin{alignat}{3}\label{eqn: baru and underaccentbaru for LV}
\bar{u}_1 & =\max\bigg(\frac{\sigma_1}{c_{11}},\frac{\sigma_2}{c_{21}}\bigg), & \qquad 
\bar{u}_2 & = \max\bigg(\frac{\sigma_1}{c_{12}},\frac{\sigma_2}{c_{22}}\bigg),\\ \notag
\underaccent\bar{u}_1 & = \min\bigg(\frac{\sigma_1}{c_{11}},\frac{\sigma_2}{c_{21}}\bigg), & \qquad 
\underaccent\bar{u}_2 & = \min\bigg(\frac{\sigma_1}{c_{12}},\frac{\sigma_2}{c_{22}}\bigg),
\end{alignat}

\end{corollary}

\begin{proof}
We apply Theorem~\ref{thm: NBMP for n species} to prove Corollary~\ref{cor: NBMP for NDC-tw}. Due to \eqref{eqn: baru and underaccentbaru for LV}, it can be verified that $\mathbf{[H]}$ is satisfied. Indeed, we have
\begin{eqnarray*}
\underaccent\bar{\mathcal{R}}&=&\Bigg\{ (u_1,u_2)\;\Bigg|\; \sum_{i=1}^{2}\frac{\displaystyle u_i}{\displaystyle\min_{j=1,2}\frac{\sigma_j}{c_{ji}}}\le 1,\; u_1,u_2\ge 0 \Bigg\};\\
\bar{\mathcal{R}}&=&\Bigg\{ (u_1,u_2)\;\Bigg|\; \sum_{i=1}^{2}\frac{\displaystyle u_i}{\displaystyle\max_{j=1,2}\frac{\sigma_j}{c_{ji}}}\ge 1,\; u_1,u_2\ge 0 \Bigg\}.
\end{eqnarray*}
Since $\displaystyle \min_{j=1,2}\frac{\sigma_j}{c_{ji}}$ \Big($\displaystyle \max_{j=1,2}\frac{\sigma_j}{c_{ji}}$, respectively\Big) is the smallest (largest, respectively) $u_i$-intercept of the two planes  $\sigma_i-c_{i1}\,u_1-c_{i2}\,u_2=0$ $(i=1,2)$, we see that 
\begin{eqnarray*}
\sigma_i-c_{i1}\,u_1-c_{i2}\,u_2\ge 0 &\text{ whenever } (u_1,u_2)\in \underaccent\bar{\mathcal{R}};\\\\
\sigma_i-c_{i1}\,u_1-c_{i2}\,u_2\le 0 &\text{ whenever } (u_1,u_2)\in \bar{\mathcal{R}},
\end{eqnarray*}
for each $i=1,2$. The desired result follows from Theorem~\ref{thm: NBMP for n species}.

\end{proof}

As an interesting application of the linear diffusion NBMP (Theorem~\ref{thm: NBMP for m=1}), we investigate the situation where one exotic competing species (say, $w$) invades the ecological system of two native species (say, $u$ and $v$) that are competing in the absence of $w$. A problem related to \textit{competitive exclusion} (\cite{Armstrong80Competitive-exclusion,Hsu08Competitive-exclusion,Hsu-Smith-Waltman96Competitive-exclusion-coexistence-Competitive,Jang13Competitive-exclusion-Leslie-Gower-competition-Allee,McGehee77Competitive-exclusion,Smith94Competition}) or \textit{competitor-mediated coexistence} (\cite{CantrellWard97Competition-mediatedCoexistence,Kastendiek82Competitor-mediatedCoexistence3Species,Mimura15DynamicCoexistence3species}) then arises. The Lotka-Volterra system of three competing species is usually used to model this situation (\cite{Adamson12SpeciesCyclicCompetition,EiMimuraIkota99SegregatingCompetition-diffusion,Grossberg78Decision-Patterns-Oscillations-LVcompetitive,Gyllenberg09LV3species-Heteroclinic,Hallam79Persistence-Extinction3speciesLV,KoRyuAhn14Coexistence3Competing-species,Maier13Integration3-dimensionalLV,Mimura15DynamicCoexistence3species,PetrovskiiShigesada01Spatio-temporalThree-competitive-species,Zeeman98Three-dimensionalCompetitiveLV,Zeeman93Hopf-bifurcationsCompetitive3speciesLV}). Under this situation, the traveling wave solution $(u(x),v(x),w(x))$ satisfies the following system:

\begin{equation}\label{eqn: 3 species TWS nonexistence linear diffusion}
\begin{cases}
\vspace{3mm} 
d_1\,u_{xx}+\theta\,u_x+u\,(\sigma_1-c_{11}\,u-c_{12}\,v-c_{13}\,w)=0,\ \ &x\in\mathbb{R},\\
\vspace{3mm} 
d_2\,v_{xx}+\theta\,v_x+v\,(\sigma_2-c_{21}\,u-c_{22}\,v-c_{23}\,w)=0,\ \ &x\in\mathbb{R},\\
d_3\,w_{xx}+\theta\,w_x+w\,(\sigma_3-c_{31}\,u-c_{32}\,v-c_{33}\,w)=0,\ \ &x\in\mathbb{R}.
\end{cases}
\end{equation}
Clearly, \eqref{eqn: degenerate autonomous system of n species} includes \eqref{eqn: 3 species TWS nonexistence linear diffusion} as a special case. For \eqref{eqn: 3 species TWS nonexistence linear diffusion}, existence of solutions with profiles of one-hump waves coupled with the boundary conditions
\begin{equation}\label{eqn: 3 species TWS nonexistence linear diffusion BC}
(u,v,w)(-\infty)=\Big(\frac{\sigma_1}{c_{11}},0,0\Big), \ \ (u,v,w)(\infty)=\Big(0,\frac{\sigma_2}{c_{22}},0\Big).
\end{equation}
is investigated under certain assumptions on the parameters by finding exact solutions (\cite{CHMU,CPAA-16}) and using the numerical tracking method AUTO (\cite{CHMU}). A one-hump wave is referred to as a traveling wave consisting of a forward front $v$, a backward front $u$, and a pulse $w$ in the middle. On the other hand, nonexistence of solutions for the problem \eqref{eqn: 3 species TWS nonexistence linear diffusion}, \eqref{eqn: 3 species TWS nonexistence linear diffusion BC} is established by means of the NBMP (Theorem~\ref{thm: NBMP for m=1}) as well as the elliptic maximum principle under certain conditions (\cite{CPAA-16,JDE-16}).

Recently, new dynamical patterns exhibited by the solutions of the Lotka-Volterra system of three competing species have been found in \cite{Mimura15DynamicCoexistence3species}, where traveling wave solutions of the three species (i.e. solutions of \eqref{eqn: 3 species TWS nonexistence linear diffusion} are used as building blocks \eqref{eqn: 3 species TWS nonexistence linear diffusion} to generate dynamical patterns in which three species coexist. This numerical evidence demonstrates (indicates) from the viewpoint of dynamical coexistence of the three species the great importance of the one-hump waves in the problem \eqref{eqn: 3 species TWS nonexistence linear diffusion}, \eqref{eqn: 3 species TWS nonexistence linear diffusion BC}.

The linear diffusion terms in \eqref{eqn: 3 species TWS nonexistence linear diffusion} are based on Fick's law in which the population flux is proportional to the gradient of the population density. In some situations, however, evidences from field studies have shown the inadequacy of this model. Due to population pressure, the phenomenon that species tend to avoid crowded can be characterized by the population flux which depends on both the population density and its gradient (\cite{Murray93MathBio,Sherratt-Marchant96NonsharpTWdegenerate-nonlinear-diffusion,Witelski95MergingTW-porous-Fisher}). Gurney and Nisbet considered the nonlinear diffusion effect described above, and proposed the following the model (\cite{Gurney-Nisbet75inhomogeneous-populations,Gurney-Nisbet76non-linear-population}) 
\begin{equation}\label{eqn: porous medium equation}
u_t=(u\,u_x)_x+u\,(u-1),
\end{equation}
where the population flux is proportional to $u$ and $u_x$. Based on porous medium version of the Fisher equation \eqref{eqn: porous medium equation} (\cite{Pablo-Sanchez98Porous-FisherTW,Sanchez-Maini95Degenerate-RD-TW,Sanchez-Maini94SharpTW-degenerate-Fisher}), \eqref{eqn: 3 species TWS nonexistence linear diffusion} becomes
\begin{equation}\label{eqn: 3 species TWS nonexistence}
\begin{cases}
\vspace{3mm} 
d_1\,(u^2)_{xx}+\theta\,u_x+u\,(\sigma_1-c_{11}\,u-c_{12}\,v-c_{13}\,w)=0,\ \ &x\in\mathbb{R},\\
\vspace{3mm} 
d_2\,(v^2)_{xx}+\theta\,v_x+v\,(\sigma_2-c_{21}\,u-c_{22}\,v-c_{23}\,w)=0,\ \ &x\in\mathbb{R},\\
d_3\,(w^2)_{xx}+\theta\,w_x+w\,(\sigma_3-c_{31}\,u-c_{32}\,v-c_{33}\,w)=0,\ \ &x\in\mathbb{R},\\
\end{cases}
\end{equation}
For the existence of solutions of the problem \eqref{eqn: 3 species TWS nonexistence}, \eqref{eqn: 3 species TWS nonexistence linear diffusion BC}, it seems as far as we know, not available in the literature. As a starting point to study this problem, we instead find the conditions on the parameters under which the solutions do not exist. With the aid of the NBMP for the problem \eqref{eqn: 3 species TWS nonexistence}, \eqref{eqn: 3 species TWS nonexistence linear diffusion BC}, this can be achieved as the following nonexistence result shows.    

\begin{theorem}[\textbf{Nonexistence of three species waves}]\label{thm: Nonexistence 3 species}
Under either \textit{(i)} or \textit{(ii)}, \eqref{eqn: 3 species TWS nonexistence} admits no positive solution $(u(x),v(x),w(x))$ with $u(x)$,$v(x)$,$w(x)\not\equiv$ constant.
\begin{itemize}
  \item [(i)] Let $\phi_1=\sigma_1-c_{13}\,\sigma_3\,c_{33}^{-1}$ and $\phi_2=\sigma_2-c_{23}\,\sigma_3\,c_{33}^{-1}$. Assume that the following hypotheses hold:
\begin{itemize}
\item [$\mathbf{[H0]}$] $(u,v)(\pm\infty)\neq (0,0)$;
\item [$\mathbf{[H1]}$] $\displaystyle\max_{x\in\mathbb{R}} w(x)=w(x_0)$ for some $x_0\in\mathbb{R}$; 
\item [$\mathbf{[H2]}$] $\phi_1,\phi_2>0$; 
\item [$\mathbf{[H3]}$] $\lambda_{\ast}:=d_1\,d_2\,\underaccent\bar{u}_{\ast}\,\underaccent\bar{v}_{\ast}\,\displaystyle\min \bigg(\frac{c_{31}}{d_1},\frac{c_{32}}{d_2}\bigg)\,\displaystyle\sqrt{\frac{\displaystyle c_{31}\,c_{32}}{(\displaystyle c_{31}\,d_1\,\underaccent\bar{u}_{\ast}^2+c_{32}\,d_2\,\underaccent\bar{v}_{\ast}^2)\,(c_{31}\,d_2+c_{32}\,d_1)}}\ge\sigma_3$,
where
\begin{alignat}{3}
\underaccent\bar{u}_{\ast} & = \min\bigg(\frac{\phi_1}{c_{11}},\frac{\phi_2}{c_{21}}\bigg), & \qquad 
\underaccent\bar{v}_{\ast} & = \min\bigg(\frac{\phi_1}{c_{12}},\frac{\phi_2}{c_{22}}\bigg).
\end{alignat}

\end{itemize}
  \item [(ii)]  Assume that the following hypotheses hold:
\begin{itemize}
\item [$\mathbf{[H4]}$] $\displaystyle\min_{x\in\mathbb{R}} w(x)=w(x_0)$ for some $x_0\in\mathbb{R}$; 
\item [$\mathbf{[H5]}$] $\lambda^{\ast}:=\displaystyle\bigg(
\frac{c_{31}}{d_1}+\frac{c_{32}}{d_2}
\bigg) 
\sqrt{
\max
\bigg(
\frac{d_1}{c_{31}},\frac{d_2}{c_{32}}
\bigg)
\max
\Big(
c_{31}\,d_1\,{\bar{u}^{\ast}}^2,c_{32}\,d_2\,{\bar{v}^{\ast}}^2
\Big)
}<\sigma_3$,
where
\begin{alignat}{3}
{\bar{u}^{\ast}} & = \max\bigg(\frac{\sigma_1}{c_{11}},\frac{\sigma_2}{c_{21}}\bigg), & \qquad 
{\bar{v}^{\ast}} & = \max\bigg(\frac{\sigma_1}{c_{12}},\frac{\sigma_2}{c_{22}}\bigg).
\end{alignat}
\item [$\mathbf{[H6]}$] $w(\pm\infty):=w_{\pm\infty}$, where either $w_{-\infty}<\displaystyle\frac{1}{c_{33}}(\sigma_3-\lambda^{\ast})$ or $w_{+\infty}<\displaystyle\frac{1}{c_{33}}(\sigma_3-\lambda^{\ast})$.
\end{itemize}
\end{itemize}
\end{theorem}
We note that when the boundary conditions are imposed at $x=\pm\infty$ like \eqref{eqn: 3 species TWS nonexistence linear diffusion BC}, hypotheses $\mathbf{[H0]}$ and $\mathbf{[H1]}$ are simultaneously satisfied. Roughly speaking, $(i)$ of Theorem~\ref{thm: Nonexistence 3 species} says from the viewpoint of ecology that when the intrinsic growth rate $\sigma_3$ of $w$ is sufficiently small ( i.e. $\mathbf{[H3]}$), the three species $u$, $v$ and $w$ cannot coexist in the ecological system modeled by \eqref{eqn: 3 species TWS nonexistence}, \eqref{eqn: 3 species TWS nonexistence linear diffusion BC}. In other words, competitor-mediated coexistence cannot occur in such a circumstance. On the other hand, $\mathbf{[H6]}$ is satisfied when the boundary conditions are 
\begin{equation}\label{eqn: 3 species TWS nonexistence linear diffusion BC (ii)}
(u,v,w)(-\infty)=\Big(\frac{\sigma_1}{c_{11}},0,0\Big), \ \ (u,v,w)(\infty)=\Big(0,\tilde{v},\tilde{w}\Big),
\end{equation}
where $v=\tilde{v}$, $w=\tilde{w}$ solves
\begin{equation}
\sigma_2-c_{22}\,v-c_{23}\,w=0, \ \ \sigma_3-c_{32}\,v-c_{33}\,w=0
\end{equation}
or 
\begin{equation}
\tilde{v}=\frac{c_{23}\,\sigma _3-c_{33}\,\sigma _2}{c_{23}\,c_{32}-c_{22}\,c_{33}}, \ \ 
\tilde{w}=\frac{c_{32}\,\sigma _2-c_{22}\,\sigma _3}{c_{23}\,c_{32}-c_{22}\,c_{33}},
\end{equation}
whenever the coexistence state $(\tilde{v},\tilde{w})$ exists. $\mathbf{[H4]}$ is an extra hypothesis on the  profile of the wave. As a consequence, $(ii)$ of Theorem~\ref{thm: Nonexistence 3 species} asserts that under certain conditions on the boundary conditions (i.e. $\mathbf{[H6]}$) and on the profile of the wave (i.e. $\mathbf{[H4]}$), coexistence among the three species $u$, $v$ and $w$ cannot occur when the intrinsic growth rate $\sigma_3$ of $w$ is sufficiently large (i.e. $\mathbf{[H5]}$).

The remainder of this paper is organized as follows. Section~\ref{sec:proof of NBMP} is devoted to the proof of Theorem~\ref{thm: NBMP for n species}. As an application of Theorem~\ref{thm: NBMP for n species}, we show in Section~\ref{sec: nonexistence} the nonexistence result of three species in Theorem~\ref{thm: Nonexistence 3 species}. In Section~\ref{sec: Concluding Remarks}, we propose some open problems concerning the NBMP. Finally, some exact traveling wave solutions and the solutions of a system of algebraic equations needed in the proof of Theorem~\ref{thm: NBMP for n species} are given in the Appendix (Section~\ref{sec: appendix}).

\section{Proof of Theorem~\ref{thm: NBMP for n species}}\label{sec:proof of NBMP}

\begin{prop} [\textbf{Lower bound in NBMP}]\label{prop: lower bed}
Suppose that $u_i(x)\in C^2(\mathbb{R})$ with $u_i(x)\ge0$ $(i=1,2,\cdots,n)$ and satisfy the following differential inequalities and asymptotic behavior:
\begin{equation*}
\textbf{(BVP-u)}
\begin{cases}
\vspace{3mm}
d_i\,(u_i^m)_{xx}+\theta\,(u_i)_{x}+u_i^{l_i}\,f_i(u_1,u_2,\cdots,u_n)\le0, \quad x\in\mathbb{R}, \quad i=1,2,\cdots,n, \\
(u_1,u_2,\cdots,u_n)(-\infty)=\textbf{e}_{-},\quad (u_1,u_2,\cdots,u_n)(\infty)=\textbf{e}_{+},
\end{cases}
\end{equation*}
where $\textbf{e}_{-}$ and $\textbf{e}_{+}$ are given by \eqref{eqn: e- and e+}. If the hypothesis 
\begin{itemize}
\item [$\mathbf{[\underline{H}]}$]
For $i=1,2,\cdots,n$, there exist $\underaccent\bar{u}_i>0$  such that
\begin{eqnarray*}
f_i(u_1,u_2,\cdots,u_n)\geq 0 &\text{ whenever} (u_1,u_2,\cdots,u_n)\in \underaccent\bar{\mathcal{R}},
\end{eqnarray*}
where $\underaccent\bar{\mathcal{R}}$ is as defined in $\mathbf{[H]}$
\end{itemize}
holds, then we have for any $\alpha_i>0$ $(i=1,2,...,n)$
\begin{equation}\label{eqn: lower bound of p}
\sum_{i=1}^{n} \alpha_i\,u_i(x)\geq 
\sqrt[m]{
\Bigg(\sum_{i=1}^{n}\frac{1}{\sqrt[m-1]{\alpha_i\,d_i\,\underaccent\bar{u}_i^m}}\Bigg)^{1-m}
\Bigg(\sum_{i=1}^{n}\frac{\alpha_i}{\sqrt[m-1]{d_i}}\Bigg)^{1-m} 
\bigg(\min_{1\le i \le n} \frac{\alpha_i^{m-1}}{d_i}\bigg)^{2}
}
\chi
\end{equation}
where $\chi$ is defined as in \eqref{eqn: chi 1 or 0}.
\end{prop}

\begin{proof}
For the case where $\textbf{e}_{+}=(0,\cdots,0)$ or $\textbf{e}_{-}=(0,\cdots,0)$, a trivial lower bound of $\displaystyle\sum_{i=1}^{n} \alpha_i\,u_i(x)$ is $0$. It suffices to show \eqref{eqn: lower bound of p} for the case $\text{\bf e}_{+}\neq(0,...,0)$ and $\text{\bf e}_{-}\neq(0,...,0)$. To this end, we let
\begin{eqnarray}\label{eqn: definition of p and q}
p(x) & = & \sum_{i=1}^{n} \alpha_i\,u_i(x);\\
q(x) & = & \sum_{i=1}^{n} \alpha_i\,d_i\,u_i^m(x).
\end{eqnarray}
Adding the $n$ equations in \textbf{(BVP-u)}, we obtain a single equation involving $p(x)$ and $q(x)$
\begin{equation}\label{eqn: ODE for p and q}
\frac{d^2 q(x)}{dx^2}+\theta\,\frac{d p(x)}{dx}+F(u_1(x),u_2(x),\cdots,u_n(x))\le0,\quad x\in\mathbb{R},
\end{equation}
where $F(u_1,u_2,\cdots,u_n):=\displaystyle\sum_{i=1}^{n} \alpha_i\,u_i^{l_i}\,f_i(u_1,u_2,\cdots,u_n)$. First of all, we show how to construct \textit{the N-barrier}.


Determining an appropriate N-barrier is crucial in establishing \eqref{eqn: lower bound of p}. The construction of the N-barrier consists of determining the positive parameters $\lambda_1$, $\lambda_2$, $\eta_1$ and $\eta_2$ such that the two hyper-ellipsoids $\displaystyle\sum_{i=1}^{n} \alpha_i\,d_i\,u_i^m=\lambda_1$ and $\displaystyle\sum_{i=1}^{n} \alpha_i\,d_i\,u_i^m=\lambda_2$, and the two hyperplanes $\displaystyle\sum_{i=1}^{n} \alpha_i\,u_i=\eta_1$ and $\displaystyle\sum_{i=1}^{n} \alpha_i\,u_i=\eta_2$ satisfy the relationship
\begin{equation}\label{eqn: Q1<P<Q2<R} 
\mathcal{P}_{\eta_2}\subset \mathcal{Q}_{\lambda_2}\subset \mathcal{P}_{\eta_1}\subset \mathcal{Q}_{\lambda_1}\subset \underaccent\bar{\mathcal{R}},
\end{equation}
where 
\begin{eqnarray}
\mathcal{P}_{\eta} & = & \Big\{ (u_1,u_2,\cdots,u_n) \;\Big|\; \sum_{i=1}^{n} \alpha_i\,u_i\le\eta,\; u_1,u_2,\cdots,u_n\ge 0\Big\}; \\
\mathcal{Q}_{\lambda} & = & \Big\{ (u_1,u_2,\cdots,u_n) \;\Big|\; \sum_{i=1}^{n} \alpha_i\,d_i\,u_i^m\le\lambda,\; u_1,u_2,\cdots,u_n\ge 0\Big\}.
\end{eqnarray}
The hyper-ellipsoids $\displaystyle\sum_{i=1}^{n} \alpha_i\,d_i\,u_i^m=\lambda_1$ and $\displaystyle\sum_{i=1}^{n} \alpha_i\,d_i\,u_i^m=\lambda_2$, and the hyperplane $\displaystyle\sum_{i=1}^{n} \alpha_i\,u_i=\eta_1$ form the N-barrier; it turns out that the hyperplane $\displaystyle\sum_{i=1}^{n} \alpha_i\,u_i=\eta_2$ determines a lower bound of $p(x)$. We follow the three steps below to construct the N-barrier:

\begin{enumerate}
  \item Let the hyperplane $\displaystyle\sum_{i=1}^{n}\frac{\displaystyle u_i}{\displaystyle\underaccent\bar{u}_i}=1$ be tangent to the hyper-ellipsoid $\displaystyle\sum_{i=1}^{n} \alpha_i\,d_i\,u_i^m=\lambda_1$ at $(u_1,u_2,\cdots,u_n)$ with $u_1,u_2,\cdots,u_n>0$ such that $\mathcal{Q}_{\lambda_1}\subset \underaccent\bar{\mathcal{R}}$. This leads to the following equations:
\begin{eqnarray}
\alpha_i\,d_i\,u_i^{m-1}\,\underaccent\bar{u}_i & = & \alpha_j\,d_j\,u_j^{m-1}\,\underaccent\bar{u}_j,\ \ i,j= 1,2,\cdots,n;\\
\sum_{i=1}^{n}\frac{\displaystyle u_i}{\displaystyle\underaccent\bar{u}_i} & = &1;\\
\sum_{i=1}^{n} \alpha_i\,d_i\,u_i^m & = & \lambda_1.
\end{eqnarray}
By Lemma~\ref{lem: lambda1} (see Section~\ref{sec: appendix}), $\lambda_1$ is determined by
\begin{equation}\label{eqn: lambda1 lower bound}
\lambda_1=\Bigg(\sum_{i=1}^{n}\frac{1}{\sqrt[m-1]{\alpha_i\,d_i\,\underaccent\bar{u}_i^m}}\Bigg)^{1-m}.
\end{equation}
  \item Setting 
\begin{equation}\label{eqn: eta1 lower bound}
\eta_1=\sqrt[m]{\lambda_1\displaystyle \min_{1\le i \le n} \frac{\alpha_i^{m-1}}{d_i}}, 
\end{equation}
the hyperplane $\displaystyle\sum_{i=1}^{n} \alpha_i\,u_i=\eta_1$ has the $n$ intercepts $\Big(\displaystyle\frac{\eta_1}{\alpha_1},0,\cdots,0\Big)$, $\Big(0,\displaystyle\frac{\eta_1}{\alpha_2},0,\cdots,0\Big)$,$\cdots$, and $\Big(0,0,\cdots,0,\displaystyle\frac{\eta_1}{\alpha_n}\Big)$ and the hyper-ellipsoid $\displaystyle\sum_{i=1}^{n} \alpha_i\,d_i\,u_i^m=\lambda_1$ has the $n$ intercepts $\bigg(\displaystyle\sqrt[m]{\frac{\lambda_1}{\alpha_1\,d_1}},0,\cdots,0\bigg)$, $\bigg(0,\displaystyle\sqrt[m]{\frac{\lambda_1}{\alpha_2\,d_2}},0,\cdots,0\bigg)$,$\cdots$, and $\bigg(0,0,\cdots,0,\displaystyle\sqrt[m]{\frac{\lambda_1}{\alpha_n\,d_n}}\bigg)$. It is easy to verify that $\mathcal{P}_{\eta_1}\subset \mathcal{Q}_{\lambda_1}$ since $\displaystyle\frac{\eta_1}{\alpha_j}\le\displaystyle\sqrt[m]{\frac{\lambda_1}{\alpha_j\,d_j}}$ for $j=1,2,\cdots,n$. Indeed, we have
\begin{align}
\label{}
\frac{\eta_1}{\alpha_j}
=&\bigg(\min_{1\le i \le n} \frac{\lambda_1\,\alpha_i^{m-1}}{d_i}\bigg)^{\frac{1}{m}}
\frac{1}{\alpha_j}
\\ \notag
\le&\bigg(\frac{\lambda_1\,\alpha_j^{m-1}}{d_j\,\alpha_j^m}\bigg)^{\frac{1}{m}} 
=\bigg(\frac{\lambda_1}{\alpha_j\,d_j}\bigg)^{\frac{1}{m}}.
\end{align}
  \item Let the hyperplane $\displaystyle\sum_{i=1}^{n} \alpha_i\,u_i=\eta_1$ be tangent to the hyper-ellipsoid $\displaystyle\sum_{i=1}^{n} \alpha_i\,d_i\,u_i^m=\lambda_2$ at $(u_1,u_2,\cdots,u_n)$ with $u_1,u_2,\cdots,u_n>0$ such that $ \mathcal{Q}_{\lambda_2}\subset \mathcal{P}_{\eta_1}$. This leads to the following equations:
\begin{eqnarray}
d_i\,u_i^{m-1} & = & d_j\,u_j^{m-1} ,\ \ i,j= 1,2,\cdots,n;\\ 
\sum_{i=1}^{n} \alpha_i\,u_i & = & \eta_1;\\
\sum_{i=1}^{n} \alpha_i\,d_i\,u_i^m & = & \lambda_2.
\end{eqnarray}
Employing Lemma~\ref{lem: lambda2} in Section~\ref{sec: appendix}, we obtain
\begin{equation}\label{eqn: lambda2 lower bound}
\lambda_2=\eta_1^{m}\,\Bigg(\sum_{i=1}^{n}\frac{\alpha_i}{\sqrt[m-1]{d_i}}\Bigg)^{1-m}.
\end{equation}

\end{enumerate}
Steps (i)$\sim$(iii) complete the construction of the N-barrier. As in step (ii), we determine $\eta_2$ by
\begin{equation}\label{eqn: eta2 lower bound}
\eta_2=\sqrt[m]{\lambda_2\displaystyle \min_{1\le i \le n} \frac{\alpha_i^{m-1}}{d_i}}
\end{equation}
such that $\mathcal{P}_{\eta_2}\subset \mathcal{Q}_{\lambda_2}$. From \eqref{eqn: lambda1 lower bound}, \eqref{eqn: eta1 lower bound}, \eqref{eqn: lambda2 lower bound} and \eqref{eqn: eta2 lower bound}, it follows immediately that $\eta_2$ is given by
\begin{align}
\label{}
\eta_2 
&= \lambda_2^{\frac{1}{m}} \bigg(\min_{1\le i \le n} \frac{\alpha_i^{m-1}}{d_i}\bigg)^{\frac{1}{m}}  
  = \eta_1\,\Bigg(\sum_{i=1}^{n}\frac{\alpha_i}{\sqrt[m-1]{d_i}}\Bigg)^{\frac{1-m}{m}}
     \bigg(\min_{1\le i \le n} \frac{\alpha_i^{m-1}}{d_i}\bigg)^{\frac{1}{m}}\\ \notag
&= \lambda_1^{\frac{1}{m}}\Bigg(\sum_{i=1}^{n}\frac{\alpha_i}{\sqrt[m-1]{d_i}}\Bigg)^{\frac{1-m}{m}}
     \bigg(\min_{1\le i \le n} \frac{\alpha_i^{m-1}}{d_i}\bigg)^{\frac{2}{m}} \\ \notag    
&=\Bigg(\sum_{i=1}^{n}\frac{1}{\sqrt[m-1]{\alpha_i\,d_i\,\underaccent\bar{u}_i^m}}\Bigg)^{\frac{1-m}{m}}
     \Bigg(\sum_{i=1}^{n}\frac{\alpha_i}{\sqrt[m-1]{d_i}}\Bigg)^{\frac{1-m}{m}} 
     \bigg(\min_{1\le i \le n} \frac{\alpha_i^{m-1}}{d_i}\bigg)^{\frac{2}{m}}.     
\end{align}

(The construction of the N-barrier for the simplified case $m=n=2$ is illustrated in Remark~\ref{rem: N-barrier for lower bounds}, which provides an intuitive idea of the construction of the N-barrier in higher dimensional cases.)

We claim that $q(x)\ge \lambda_2$, $x\in\mathbb{R}$. This proves \eqref{eqn: lower bound of p}, i.e $q(x)\ge \eta_2$, $x\in\mathbb{R}$ since the $\alpha_i>0$ $(i=1,2,\cdots,n)$ are arbitrary and the relationship $\mathcal{P}_{\eta_2}\subset \mathcal{Q}_{\lambda_2}$ holds. Now we prove the claim by contradiction. Suppose that, contrary to our claim, there exists $z\in\mathbb{R}$ such that $q(z)<\lambda_2$. Since $u_i(x)\in C^2(\mathbb{R})$ and  $(u_1,u_2,\cdots,u_n)(\pm\infty)=\textbf{e}_{\pm}$, we may assume $\displaystyle\min_{x\in\mathbb{R}} q(x)=q(z)$. 
We denote respectively by $z_2$ and $z_1$ the first points at which the solution $(u_1(x),u_2(x),\cdots,u_n(x))$ intersects the hyper-ellipsoid $\displaystyle\sum_{i=1}^{n} \alpha_i\,d_i\,u_i^m=\lambda_1$ when $x$ moves from $z$ towards $\infty$ and $-\infty$. 
For the case where $\theta\leq0$, we integrate \eqref{eqn: ODE for p and q} with respect to $x$ from $z_1$ to $z$ and obtain
\begin{equation}\label{eqn: integrating eqn}
q'(z)-q'(z_1)+\theta\,(p(z)-p(z_1))+\int_{z_1}^{z}F(u_1(x),u_2(x),\cdots,u_n(x))\,dx\leq0.
\end{equation}
On the other hand we have:
\begin{itemize}
  \item $q'(z)=0$ because of $\displaystyle\min_{x\in\mathbb{R}} q(x)=q(z)$;
  \item $q(z_1)=\lambda_1$ follows from the fact that $z_1$ is on the hyper-ellipsoid $\displaystyle\sum_{i=1}^{n} \alpha_i\,d_i\,u_i^m=\lambda_1$. Since $z_1$ is the first point
  for $q(x)$ taking the value $\lambda_1$ when $x$ moves from $z$ to $-\infty$, we conclude that $q(z_1+\delta)\leq \lambda_1$ for $z-z_1>\delta>0$
  and $q'(z_1)\leq 0$;
  \item $p(z)<\eta_1$ since $z$ is below the hyperplane $\displaystyle\sum_{i=1}^{n} \alpha_i\,u_i=\eta_1$; $p(z_1)>\eta_1$ since $z_1$ is above the hyperplane $\displaystyle\sum_{i=1}^{n} \alpha_i\,u_i=\eta_1$;
  \item let $\mathcal{F}_+=\Big\{(u_1,u_2,\cdots,u_n)\,\Big|\, F(u_1,u_2,\cdots,u_n)> 0, u_1,u_2,\cdots,u_n\ge0\Big\}$. 
  Due to the fact that $(u_1(z_1),u_2(z_1),\cdots,u_n(z_1))$ is on the hyper-ellipsoid $\displaystyle\sum_{i=1}^{n} \alpha_i\,d_i\,u_i^m=\lambda_1$ and $(u_1(z),u_2(z),\cdots,u_n(z))$ $\in\mathcal{Q}_{\lambda_2}$, $(u_1(z_1),u_2(z_1),\cdots,u_n(z_1))$, we have $(u_1(z),u_2(z),\cdots,u_n(z))$ $\in\underaccent\bar{\mathcal{R}}$ by \eqref{eqn: Q1<P<Q2<R}. Because of $\mathbf{[\underline{H}]}$ and $F(u_1,u_2,\cdots,u_n)=\displaystyle\sum_{i=1}^{n} \alpha_i\,u_i^{l_i}\,f_i(u_1,u_2,\cdots,u_n)$, it is easy to see that 
\begin{equation}
\Big\{(u_1(x),u_2(x),\cdots,u_n(x))\,\Big|\,z_1\le x \le z\Big\} \subset\underaccent\bar{\mathcal{R}} \subset \mathcal{F}_+.
\end{equation}  
Therefore we have $\displaystyle\int_{z_1}^{z}F(u_1(x),u_2(x),\cdots,u_n(x))\,dx>0$.
\end{itemize}

Combining the above arguments, we obtain
\begin{equation}
q'(z)-q'(z_1)+\theta\,(p(z)-p(z_1))+\int_{z_1}^{z}F(u_1(x),u_2(x),\cdots,u_n(x))\,dx>0,
\end{equation}
which contradicts \eqref{eqn: integrating eqn}. Therefore when $\theta\leq0$, $q(x)\geq \lambda_2$ for $x\in \mathbb{R}$. For the case where $\theta\geq0$, integrating \eqref{eqn: ODE for p and q} with respect to $x$ from $z$ to $z_2$ yields
\begin{equation}\label{eqn: eqn by integrate from z to z2}
q'(z_2)-q'(z)+\theta\,(p(z_2)-p(z))+\int_{z}^{z_2}F(u_1(x),u_2(x),\cdots,u_n(x))\,dx\leq0.
\end{equation}
In a similar manner, it can be shown that $q'(z_2)\ge 0$, $q'(z)=0$, $p(z_2)>\eta$, $p(z)<\eta$, and 
\begin{equation}
\int_{z}^{z_2}F(u_1(x),u_2(x),\cdots,u_n(x))\,dx>0.
\end{equation}
These together contradict \eqref{eqn: eqn by integrate from z to z2}. Consequently, \eqref{eqn: lower bound of p} is proved and the proof is completed.


\end{proof}


\begin{remark}[\textbf{N-barrier for lower bounds}]\label{rem: N-barrier for lower bounds}
When $\sigma_1=\sigma_2=c_{11}=c_{22}=1$, $c_{12}=a_1$, and $c_{21}=a_2$ in \textbf{(NDC-tw)} with the asymptotic behavior $\textbf{e}_{-}=(1,0)$ and $\textbf{e}_{+}=(0,1)$, we are led to the problem 
\begin{equation}\label{eqn: 2 species LV rem}
\begin{cases}
\vspace{3mm} 
d_1\,(u^2)_{xx}+\theta\,u_x+u\,(1-u-a_1\,v)=0,\ \ &x\in\mathbb{R},\\
\vspace{3mm} 
d_2\,(v^2)_{xx}+\theta\,v_x+v\,(1-a_2\,u-v)=0,\ \ &x\in\mathbb{R},\\
(u,v)(-\infty)=(1,0),\ \ (u,v)(+\infty)=(0,1).
\end{cases}
\end{equation}
To satisfy the hypothesis $\mathbf{[H]}$, we let as in the proof of Corollary~\ref{cor: NBMP for NDC-tw}
\begin{align}
 \label{eqn: ubar rem} 
  \underaccent\bar{u} &= \min \Big(1,\frac{1}{a_2}\Big),    \\
 \label{eqn: vbar rem}
  \underaccent\bar{v} &= \min \Big(1,\frac{1}{a_1}\Big).      
\end{align}
For simplicity, we shall always assume the bistable condition $a_1,a_2>1$ for \eqref{eqn: 2 species LV rem}. This gives $\underaccent\bar{u}=\displaystyle\frac{1}{a_2}$ and $\underaccent\bar{v}=\displaystyle\frac{1}{a_1}$. We readily verify that under $a_1,a_2>1$, the quadratic curve 
\begin{equation}
F(u,v):=\alpha\,u\,(1-u-a_1\,v)+\beta\,v\,(1-a_2\,u-v)=0
\end{equation}
in the first quadrant of the $uv$-plane is a hyperbola for any $\alpha,\beta>0$ and it passes through the equilibria $(0,0)$, $(1,0)$, $(1,0)$ and $\Big(\displaystyle\frac{a_1-1}{a_1\,a_2-1},\displaystyle\frac{a_2-1}{a_1\,a_2-1}\Big)$ .

We are now in the position to follow the three steps in the proof of Proposition~\ref{prop: lower bed} to construct the N-barrier for the problem \eqref{eqn: 2 species LV rem}.

\begin{enumerate}
  \item Since the line $\displaystyle\frac{u}{\underaccent\bar{u}}+\displaystyle\frac{v}{\underaccent\bar{v}}=1$ is tangent to the ellipse $\alpha\,d_1\,u^2+\beta\,d_2\,v^2=\lambda_1$ at $(u,v)$ in the first quadrant of the $uv$-plane, this leads to the following equations:
\begin{eqnarray}
\frac{\alpha\,d_1\,u}{\beta\,d_2\,v} & = & \frac{\underaccent\bar{v}}{\underaccent\bar{u}},\\
\frac{\displaystyle u}{\displaystyle\underaccent\bar{u}}+\frac{\displaystyle v}{\displaystyle\underaccent\bar{v}} & = &1,\\
\alpha\,d_1\,u^2+\beta\,d_2\,v^2 & = & \lambda_1.
\end{eqnarray}
By Lemma~\ref{lem: lambda1} (see Section~\ref{sec: appendix}), $\lambda_1$ is given by
\begin{equation}\label{eqn: lambda1 lower bound rem}
\lambda_1=\displaystyle\frac{1}{\displaystyle\frac{1}{\alpha\,d_1\,\underaccent\bar{u}^2}+\displaystyle\frac{1}{\beta\,d_2\,\underaccent\bar{v}^2}}
=\displaystyle\frac{\alpha\,\beta\,d_1\,d_2\,\underaccent\bar{u}^2\,\underaccent\bar{v}^2}{\alpha\,d_1\,\underaccent\bar{u}^2+\beta\,d_2\,\underaccent\bar{v}^2}.
\end{equation}
  \item 
The $u$-coordinate of the $u$-intercept and the $v$-coordinate of the $v$-intercept of the ellipse $\alpha\,d_1\,u^2+\beta\,d_2\,v^2=\lambda_1$ are $\displaystyle\sqrt{\displaystyle\frac{\lambda_1}{\alpha\,d_1}}$ and $\displaystyle\sqrt{\displaystyle\frac{\lambda_1}{\beta\,d_2}}$, respectively; the $u$-coordinate of the $u$-intercept and the $v$-coordinate of the line $\eta_1=\alpha\,u+\beta\,v$ are $\displaystyle\frac{\eta_1}{\alpha}$ and $\displaystyle\frac{\eta_1}{\beta}$, respectively. Because of 
\begin{equation}\label{eqn: eta1 lower bound rem}
\eta_1=\sqrt{\lambda_1\,\min \Big(\frac{\alpha}{d_1},\frac{\beta}{d_2}\Big)}, 
\end{equation}

\begin{itemize}
  \item \underline{when $\displaystyle\min \Big(\frac{\alpha}{d_1},\frac{\beta}{d_2}\Big)=\frac{\alpha}{d_1}$}, we clearly have
  \begin{align}
\label{}
  \frac{\eta_1}{\alpha} & =  \frac{1}{\alpha}\sqrt{\frac{\lambda_1\,\alpha}{d_1}}=\sqrt{\frac{\displaystyle\lambda_1}{\displaystyle\alpha\,d_1}}, \\
  \frac{\eta_1}{\beta}  & = \frac{1}{\beta}\sqrt{\frac{\lambda_1\,\alpha}{d_1}}\le \frac{\sqrt{\lambda_1}}{\beta}\sqrt{\frac{\beta}{d_2}}=\sqrt{\frac{\displaystyle\lambda_1}{\displaystyle\beta\,d_2}};
\end{align}
  \item \underline{when $\displaystyle\min \Big(\frac{\alpha}{d_1},\frac{\beta}{d_2}\Big)=\frac{\beta}{d_2}$}, we clearly have
  \begin{align}
\label{}
  \frac{\eta_1}{\alpha}  & = \frac{1}{\alpha}\sqrt{\frac{\lambda_1\,\beta}{d_2}}\le \frac{\sqrt{\lambda_1}}{\alpha}\sqrt{\frac{\alpha}{d_1}}=\sqrt{\frac{\displaystyle\lambda_1}{\displaystyle\alpha\,d_1}}, \\
  \frac{\eta_1}{\beta} & =  \frac{1}{\beta}\sqrt{\frac{\lambda_1\,\beta}{d_2}}=\sqrt{\frac{\displaystyle\lambda_1}{\displaystyle\beta\,d_2}}.
\end{align}
\end{itemize}
This means that when $\displaystyle\min \Big(\frac{\alpha}{d_1},\frac{\beta}{d_2}\Big)=\frac{\alpha}{d_1}$, the ellipse $\alpha\,d_1\,u^2+\beta\,d_2\,v^2=\lambda_1$ and the line $\eta_1=\alpha\,u+\beta\,v$ possesess the same $u$-coordinate of the $u$-intercept, i.e. $\displaystyle\sqrt{\displaystyle\frac{\lambda_1}{\alpha\,d_1}}=\displaystyle\frac{\eta_1}{\alpha}$; meanwhile, the inequality $\displaystyle\frac{\eta_1}{\beta}\le\sqrt{\frac{\displaystyle\lambda_1}{\displaystyle\beta\,d_2}}$ indicates that the $v$-coordinate of the $v$-intercept of the line $\eta_1=\alpha\,u+\beta\,v$ is not larger than that of the $v$-intercept of the ellipse $\alpha\,d_1\,u^2+\beta\,d_2\,v^2=\lambda_1$. A similar conclusion can be drawn for the case of $\displaystyle\min \Big(\frac{\alpha}{d_1},\frac{\beta}{d_2}\Big)=\frac{\beta}{d_2}$. 
  \item The fact that the line $\eta_1=\alpha\,u+\beta\,v$ is tangent to the ellipse $\alpha\,d_1\,u^2+\beta\,d_2\,v^2=\lambda_2$ at $(u,v)$ in the first quadrant of the $uv$-plane yields the following equations:
\begin{eqnarray}
\frac{\alpha\,d_1\,u}{\beta\,d_2\,v} & = & \frac{\alpha}{\beta},\\
\alpha\,u+\beta\,v & = & \eta_1,\\
\alpha\,d_1\,u^2+\beta\,d_2\,v^2 & = & \lambda_2.
\end{eqnarray}
Employing Lemma~\ref{lem: lambda2} in Section~\ref{sec: appendix}, we obtain
\begin{equation}\label{eqn: lambda2 lower bound rem}
\lambda_2=\displaystyle\frac{\eta_1^2}{\displaystyle\frac{\alpha}{d_1}+\frac{\beta}{d_2}}
=\displaystyle\frac{\eta_1^2\,d_1\,d_2}{\alpha\,d_2+\beta\,d_1}.
\end{equation}
\end{enumerate}
The above three steps complete the construction of the N-barrier. Finally, we determine the line $\eta_2=\alpha\,u+\beta\,v$ by setting
\begin{equation}\label{eqn: eta2 lower bound rem}
\eta_2=\sqrt{\lambda_2\,\min \Big(\frac{\alpha}{d_1},\frac{\beta}{d_2}\Big)}
\end{equation}
such that, as in step (ii), the line $\eta_2=\alpha\,u+\beta\,v$ lies \textit{entirely below} the ellipse $\alpha\,d_1\,u^2+\beta\,d_2\,v^2=\lambda_2$ in the first quadrant of the $uv$-plane. Combining \eqref{eqn: lambda1 lower bound rem}, \eqref{eqn: eta1 lower bound rem}, \eqref{eqn: lambda2 lower bound rem} and \eqref{eqn: eta2 lower bound rem}, we arrive at 
\begin{align}
\label{}
\eta_2 
&= \sqrt{\lambda_2\,\min \Big(\frac{\alpha}{d_1},\frac{\beta}{d_2}\Big)}  
  = \eta_1\,\sqrt{\min \Big(\frac{\alpha}{d_1},\frac{\beta}{d_2}\Big)\,\frac{d_1\,d_2}{\alpha\,d_2+\beta\,d_1}}\\ \notag
&= \sqrt{\lambda_1\,\min \Big(\frac{\alpha}{d_1},\frac{\beta}{d_2}\Big)}\,\sqrt{\min \Big(\frac{\alpha}{d_1},\frac{\beta}{d_2}\Big)\,\frac{d_1\,d_2}{\alpha\,d_2+\beta\,d_1}} \\ \notag    
&= \min \Big(\frac{\alpha}{d_1},\frac{\beta}{d_2}\Big)\,\sqrt{\frac{\displaystyle \alpha\,\beta\,d_1\,d_2\,\underaccent\bar{u}^2\,\underaccent\bar{v}^2}{\displaystyle\alpha\,d_1\,\underaccent\bar{u}^2+\beta\,d_2\,\underaccent\bar{v}^2}\,\frac{d_1\,d_2}{\alpha\,d_2+\beta\,d_1}} \\ \notag
&= d_1\,d_2\,\underaccent\bar{u}\,\underaccent\bar{v}\,\min \Big(\frac{\alpha}{d_1},\frac{\beta}{d_2}\Big)\,\sqrt{\frac{\displaystyle \alpha\,\beta}{(\displaystyle\alpha\,d_1\,\underaccent\bar{u}^2+\beta\,d_2\,\underaccent\bar{v}^2)\,(\alpha\,d_2+\beta\,d_1)}}\\ \notag
&= \alpha\,\beta\,\underaccent\bar{u}\,\underaccent\bar{v}\,\min \Big(\frac{d_1}{\alpha},\frac{d_2}{\beta}\Big)\,\sqrt{\frac{\displaystyle \alpha\,\beta}{(\displaystyle\alpha\,d_1\,\underaccent\bar{u}^2+\beta\,d_2\,\underaccent\bar{v}^2)\,(\alpha\,d_2+\beta\,d_1)}}.
\end{align}
The lower bound $\eta_2 $ coincides with that given in Corollary~\ref{cor: NBMP for NDC-tw}.

It follows immediately from step (ii) that there are two conditions: $\displaystyle\min \Big(\frac{\alpha}{d_1},\frac{\beta}{d_2}\Big)=\frac{\alpha}{d_1}$ and $\displaystyle\min \Big(\frac{\alpha}{d_1},\frac{\beta}{d_2}\Big)=\frac{\beta}{d_2}$. We show the N-barrier for each condition in Figure~\ref{fig: N-barrier lower bound}: the N-barrier for the case $\displaystyle\min \Big(\frac{\alpha}{d_1},\frac{\beta}{d_2}\Big)=\frac{\alpha}{d_1}$ is shown in Figure~\ref{fig: lower bound alpha/d1<beta/d2}, while the one for the case $\displaystyle\min \Big(\frac{\alpha}{d_1},\frac{\beta}{d_2}\Big)=\frac{\beta}{d_2}$ is shown in Figure~\ref{fig: lower bound beta/d2<alpha/d1}. We note that through the example of Figure~\ref{fig: N-barrier lower bound} in which the N-barrier for the lower dimensional problem \eqref{eqn: 2 species LV rem} is constructed, the N-barrier in the hyper-space in the proof of Proposition~\ref{prop: lower bed} become immediate.

\begin{figure}[ht!]
\centering
\mbox{
\subfigure[\underline{$\displaystyle\min \Big(\frac{\alpha}{d_1},\frac{\beta}{d_2}\Big)=\frac{\alpha}{d_1}$}: $d_1=3$, $d_2=4$, $a_1=2$, $a_2=3$, $\alpha=1$, $\beta=2$, $\underline{u}=\displaystyle\frac{1}{3}$, $\underline{v}=\displaystyle\frac{1}{2}$, $\lambda_1=\displaystyle\frac{2}{7}$, $\lambda_2=\displaystyle\frac{4}{35}$, $\eta_1=\displaystyle\sqrt{\frac{2}{21}}$, $\eta_2=\displaystyle\frac{2}{\sqrt{105}}$.]{\includegraphics[width=0.53\textwidth]{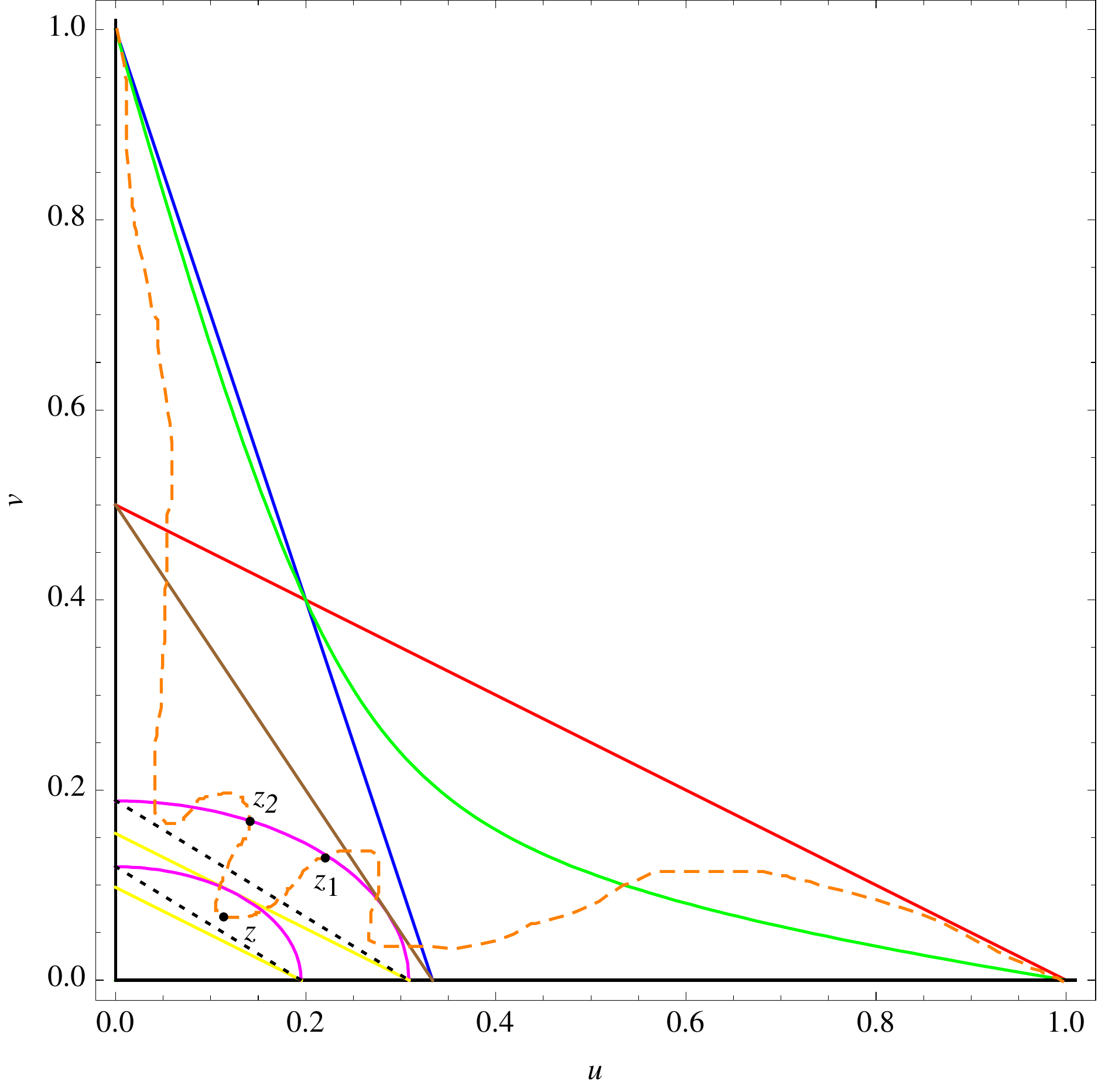}
             \label{fig: lower bound alpha/d1<beta/d2}    } \quad \hspace{0mm}
\subfigure[\underline{$\displaystyle\min \Big(\frac{\alpha}{d_1},\frac{\beta}{d_2}\Big)=\frac{\beta}{d_2}$}: $d_1=3$, $d_2=4$, $a_1=2$, $a_2=3$, $\alpha=1$, $\beta=1$, $\underline{u}=\displaystyle\frac{1}{3}$, $\underline{v}=\displaystyle\frac{1}{2}$, $\lambda_1=\displaystyle\frac{1}{4}$, $\lambda_2=\displaystyle\frac{3}{28}$, $\eta_1=\displaystyle\frac{1}{4}$, $\eta_2=\displaystyle\frac{\sqrt{\displaystyle\frac{3}{7}}}{4}$.]{\includegraphics[width=0.53\textwidth]{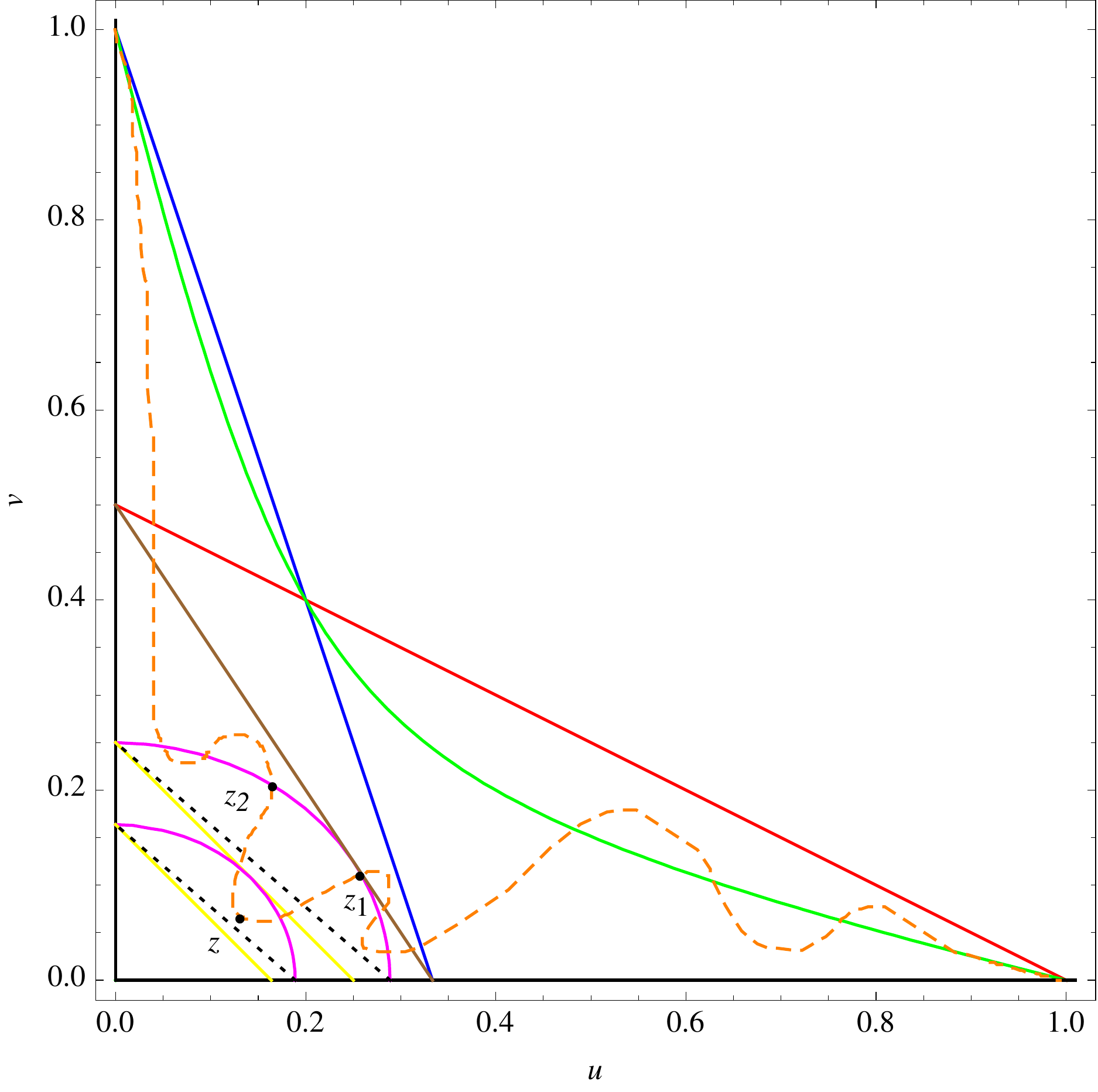}
             \label{fig: lower bound beta/d2<alpha/d1}    } \quad \hspace{0mm}
             }
\caption{\small Red line: $1-u-a_1\,v=0$; blue line: $1-a_2\,u-v=0$; green curve: $F(u,v):=\alpha\,u\,(1-u-a_1\,v)+\beta\,v\,(1-a_2\,u-v)=0$; brown line: $\displaystyle\frac{u}{\underline{u}}+\frac{v}{\underline{v}}=1$, where $\underline{u}$ and $\underline{v}$ are given by \eqref{eqn: ubar rem} and \eqref{eqn: vbar rem}  ; magenta ellipse (above): $\alpha\,d_1\,u^2+\beta\,d_2\,v^2=\lambda_1$, where $\lambda_1$ is given by \eqref{eqn: lambda1 lower bound rem}; magenta ellipse (below): $\alpha\,d_1\,u^2+\beta\,d_2\,v^2=\lambda_2$, where $\lambda_2$ is given by \eqref{eqn: lambda2 lower bound rem}; yellow line (above): $\alpha\,u+\beta\,v=\eta_1$, where $\eta_1$ is given by \eqref{eqn: eta1 lower bound rem}; yellow line (below): $\alpha\,u+\beta\,v=\eta_2$, where $\eta_2$ is given by \eqref{eqn: eta2 lower bound rem}; dashed orange curve: the solution $(u(x),v(x))$; dotted line (above): $\displaystyle\frac{u}{\sqrt{\frac{\lambda_1}{\alpha\,d_1}}}+\displaystyle\frac{v}{\sqrt{\frac{\lambda_1}{\beta\,d_2}}}=1$; dotted line (below): $\displaystyle\frac{u}{\sqrt{\frac{\lambda_2}{\alpha\,d_1}}}+\displaystyle\frac{v}{\sqrt{\frac{\lambda_2}{\beta\,d_2}}}=1$.
\label{fig: N-barrier lower bound}}
\end{figure}

\end{remark}

\begin{prop} [\textbf{Upper bound in NBMP}]\label{prop: upper bed}
Suppose that $u_i(x)\in C^2(\mathbb{R})$ with $u_i(x)\ge0$ $(i=1,2,\cdots,n)$ and satisfy the following differential inequalities and asymptotic behavior:
\begin{equation*}
\textbf{(BVP-l)}
\begin{cases}
\vspace{3mm}
d_i\,(u_i^m)_{xx}+\theta\,(u_i)_{x}+u_i^{l_i}\,f_i(u_1,u_2,\cdots,u_n)\ge0, \quad x\in\mathbb{R}, \quad i=1,2,\cdots,n, \\
(u_1,u_2,\cdots,u_n)(-\infty)=\textbf{e}_{-},\quad (u_1,u_2,\cdots,u_n)(\infty)=\textbf{e}_{+},
\end{cases}
\end{equation*}
where $\textbf{e}_{-}$ and $\textbf{e}_{+}$ are given by \eqref{eqn: e- and e+}. If the hypothesis 
\begin{itemize}
\item [$\mathbf{[\bar{H}]}$]
For $i=1,2,\cdots,n$, there exist $\bar{u}_i>0$  such that
\begin{eqnarray*}
f_i(u_1,u_2,\cdots,u_n)\le 0 &\text{ whenever } (u_1,u_2,\cdots,u_n)\in\bar{\mathcal{R}},
\end{eqnarray*}
where $\bar{\mathcal{R}}$ is as defined in $\mathbf{[H]}$
\end{itemize}
holds, then we have for any $\alpha_i>0$ $(i=1,2,\cdots,n)$
\begin{equation}\label{eqn: upper bound of p}
\sum_{i=1}^{n} \alpha_i\,u_i(x)\leq 
\sqrt[m]{
\Bigg(\sum_{i=1}^{n}\frac{\alpha_i}{\sqrt[m-1]{d_i}}\Bigg)^{2\,(m-1)} 
\bigg(\max_{1\le i \le n} \frac{d_i}{\alpha_i^{m-1}}\bigg)
\Bigg(\max_{1\le i \le n}\alpha_i\,d_i\,\bar{u}_i^m \Bigg)
}
\end{equation}
\end{prop}

\begin{proof}
We show by employing the N-barrier method as in the proof of Proposition~\ref{prop: lower bed} the upper bound given by \eqref{eqn: upper bound of p}. The construction of an appropriate N-barrier is the main ingredient of our proof. To do this, let
\begin{eqnarray}
\mathcal{P}_{\eta} & = & \Big\{ (u_1,u_2,\cdots,u_n) \;\Big|\; \sum_{i=1}^{n} \alpha_i\,u_i\ge\eta,\; u_1,u_2,\cdots,u_n\ge 0\Big\}; \\
\mathcal{Q}_{\lambda} & = & \Big\{ (u_1,u_2,\cdots,u_n) \;\Big|\; \sum_{i=1}^{n} \alpha_i\,d_i\,u_i\ge\lambda,\; u_1,u_2,\cdots,u_n\ge 0\Big\}.
\end{eqnarray}
Recall \eqref{eqn: definition of p and q} in the proof of Proposition~\ref{prop: lower bed}. Adding the $n$ equations in \textbf{(BVP-l)}, we obtain the equation
\begin{equation}\label{eqn: ODE for p and q upper bound}
\frac{d^2 q(x)}{dx^2}+\theta\,\frac{d p(x)}{dx}+F(u_1(x),u_2(x),\cdots,u_n(x))\ge0,\quad x\in\mathbb{R},
\end{equation}
where $F(u_1,u_2,\cdots,u_n):=\displaystyle\sum_{i=1}^{n} \alpha_i\,u_i^{l_i}\,f_i(u_1,u_2,\cdots,u_n)$.

We determine the positive parameters $\lambda_1$, $\lambda_2$, $\eta_1$ and $\eta_2$ such that the two hyper-ellipsoids $\displaystyle\sum_{i=1}^{n} \alpha_i\,d_i\,u_i^m=\lambda_1$, $\displaystyle\sum_{i=1}^{n} \alpha_i\,d_i\,u_i^m=\lambda_2$, and the two hyperplanes $\displaystyle\sum_{i=1}^{n} \alpha_i\,u_i=\eta_1$, $\displaystyle\sum_{i=1}^{n} \alpha_i\,u_i=\eta_2$ satisfy the relationship
\begin{equation}\label{eqn: Q1<P<Q2<R upper bound}
\mathcal{P}_{\eta_2}\supset \mathcal{Q}_{\lambda_2}\supset \mathcal{P}_{\eta_1}\supset \mathcal{Q}_{\lambda_1}\supset \bar{\mathcal{R}}.
\end{equation}
The hyper-ellipsoids $\displaystyle\sum_{i=1}^{n} \alpha_i\,d_i\,u_i^m=\lambda_1$, $\displaystyle\sum_{i=1}^{n} \alpha_i\,d_i\,u_i^m=\lambda_2$, and the hyperplane $\displaystyle\sum_{i=1}^{n} \alpha_i\,u_i=\eta_1$ form the N-barrier and it turns out that the hyperplane $\displaystyle\sum_{i=1}^{n} \alpha_i\,u_i=\eta_2$ determines the upper bound in \eqref{eqn: upper bound of p}. We follow the three steps below to construct the N-barrier:

\begin{enumerate}
  \item Setting 
\begin{equation}\label{eqn: lambda1 upper bound}
\lambda_1=\max_{1\le i \le n}\alpha_i\,d_i\,\bar{u}_i^m, 
\end{equation}
the hyper-ellipsoid $\displaystyle\sum_{i=1}^{n} \alpha_i\,d_i\,u_i^m=\lambda_1$ has the $n$ intercepts $\bigg(\displaystyle\sqrt[m]{\frac{\lambda_1}{\alpha_1\,d_1}},0,\cdots,0\bigg)$, $\bigg(0,\displaystyle\sqrt[m]{\frac{\lambda_1}{\alpha_2\,d_2}},0,\cdots,0\bigg)$,$\cdots$, and $\bigg(0,0,\cdots,0,\displaystyle\sqrt[m]{\frac{\lambda_1}{\alpha_n\,d_n}}\bigg)$ and the hyperplane $\displaystyle\sum_{i=1}^{n}\frac{\displaystyle u_i}{\displaystyle\bar{u}_i}=1$ has the $n$ intercepts $(\bar{u}_1,0,\cdots,0)$, $(0,\bar{u}_2,0,\cdots,0)$,$\cdots$, and $(0,0,\cdots,0,\bar{u}_n)$. It is easy to verify that $\mathcal{Q}_{\lambda_1}\supset \bar{\mathcal{R}}$ since $\bar{u}_j\le\displaystyle\sqrt[m]{\frac{\lambda_1}{\alpha_j\,d_j}}$ for $j=1,2,\cdots,n$. Indeed, we have
\begin{align}
\label{}
\Big(\frac{\displaystyle\lambda_1}{\displaystyle\alpha_j\,d_j}\Big)^{\frac{1}{m}}
=&\Bigg(\frac{\displaystyle\max_{1\le i \le n}\alpha_i\,d_i\,\bar{u}_i^m}{\displaystyle\alpha_j\,d_j}\Bigg)^{\frac{1}{m}}
\\ \notag
\ge&\Bigg(\frac{\displaystyle\alpha_j\,d_j\,\bar{u}_j^m}{\displaystyle\alpha_j\,d_j}\Bigg)^{\frac{1}{m}} 
=\bar{u}_j.
\end{align}
  \item Let the hyperplane $\displaystyle\sum_{i=1}^{n} \alpha_i\,u_i=\eta_1$ be tangent to the hyper-ellipsoid $\displaystyle\sum_{i=1}^{n} \alpha_i\,d_i\,u_i^m=\lambda_1$ at $(u_1,u_2,\cdots,u_n)$ with $u_1,u_2,\cdots,u_n>0$ such that $\mathcal{P}_{\eta_1}\supset \mathcal{Q}_{\lambda_1}$. This leads to the following equations:
\begin{eqnarray}
d_i\,u_i^{m-1} & = & d_j\,u_j^{m-1} ,\ \ i,j= 1,2,\cdots,n;\\ 
\sum_{i=1}^{n} \alpha_i\,u_i & = & \eta_1;\\
\sum_{i=1}^{n} \alpha_i\,d_i\,u_i^m & = & \lambda_1.
\end{eqnarray}
Employing Lemma~\ref{lem: lambda2} in Section~\ref{sec: appendix}, we obtain
\begin{equation}\label{eqn: eta1 upper bound}
\eta_1=\lambda_1^{\frac{1}{m}}\,\Bigg(\sum_{i=1}^{n}\frac{\alpha_i}{\sqrt[m-1]{d_i}}\Bigg)^{\frac{m-1}{m}}.
\end{equation}
  \item Setting 
\begin{equation}\label{eqn: lambda2 upper bound}
\lambda_2=\eta_1^m\bigg(\max_{1\le i \le n} \frac{d_i}{\alpha_i^{m-1}}\bigg), 
\end{equation}
the hyper-ellipsoid $\displaystyle\sum_{i=1}^{n} \alpha_i\,d_i\,u_i^m=\lambda_2$ has the $n$ intercepts $\Big(\sqrt[m]{\displaystyle\frac{\displaystyle\lambda_2}{\displaystyle\alpha_1\,d_1}},0,\cdots,0\Big)$, $\Big(0,\sqrt[m]{\displaystyle\frac{\displaystyle\lambda_2}{\displaystyle\alpha_2\,d_2}},0,\cdots,0\Big)$,$\cdots$, and $\Big(0,0,\cdots,0,\sqrt[m]{\displaystyle\frac{\displaystyle\lambda_2}{\displaystyle\alpha_n\,d_n}}\Big)$ and the hyperplane $\displaystyle\sum_{i=1}^{n} \alpha_i\,u_i=\eta_1$ has the $n$ intercepts $\Big(\displaystyle\frac{\displaystyle\eta_1}{\displaystyle\alpha_1},0,\cdots,0\Big)$, $\Big(0,\displaystyle\frac{\displaystyle\eta_1}{\displaystyle\alpha_2},0,\cdots,0\Big)$,$\cdots$, and $\Big(0,0,\cdots,0,\displaystyle\frac{\displaystyle\eta_1}{\displaystyle\alpha_n}\Big)$. It is easy to verify that $\mathcal{Q}_{\lambda_2}\supset \mathcal{P}_{\eta_1}$ since $\displaystyle\frac{\displaystyle\eta_1}{\displaystyle\alpha_j}\le\Big(\displaystyle\frac{\displaystyle\lambda_2}{\displaystyle\alpha_j\,d_j}\Big)^{\frac{1}{m}}$ for $j=1,2,\cdots,n$. Indeed, we have
\begin{align}
\label{}
\Big(\frac{\displaystyle\lambda_2}{\displaystyle\alpha_j\,d_j}\Big)^{\frac{1}{m}}
=&\eta_1\Bigg(\frac{\displaystyle\max_{1\le i \le n}d_i\,\alpha_i^{1-m}}{\displaystyle\alpha_j\,d_j}\Bigg)^{\frac{1}{m}}
\\ \notag
\ge&\eta_1\Bigg(\frac{\displaystyle d_j\,\alpha_j^{1-m}}{\displaystyle\alpha_j\,d_j}\Bigg)^{\frac{1}{m}} 
=\frac{\displaystyle\eta_1}{\displaystyle\alpha_j}.
\end{align}
\end{enumerate}
Steps (i)$\sim$(iii) complete the construction of the N-barrier. As in step (ii), we determine $\eta_2$ by
letting the hyperplane $\displaystyle\sum_{i=1}^{n} \alpha_i\,u_i=\eta_2$ be tangent to the hyper-ellipsoid $\displaystyle\sum_{i=1}^{n} \alpha_i\,d_i\,u_i^m=\lambda_2$ at $(u_1,u_2,\cdots,u_n)$ with $u_1,u_2,\cdots,u_n>0$ such that $\mathcal{P}_{\eta_2}\supset \mathcal{Q}_{\lambda_2}$. This leads to the following equations:
\begin{eqnarray}
d_i\,u_i^{m-1} & = & d_j\,u_j^{m-1} ,\ \ i,j= 1,2,\cdots,n;\\ 
\sum_{i=1}^{n} \alpha_i\,u_i & = & \eta_2;\\
\sum_{i=1}^{n} \alpha_i\,d_i\,u_i^m & = & \lambda_2.
\end{eqnarray}
Employing Lemma~\ref{lem: lambda2} in Section~\ref{sec: appendix} again, we obtain
\begin{equation}\label{eqn: eta2 upper bound}
\eta_2=\lambda_2^{\frac{1}{m}}\,\Bigg(\sum_{i=1}^{n}\frac{\alpha_i}{\sqrt[m-1]{d_i}}\Bigg)^{\frac{m-1}{m}}.
\end{equation}
such that $\mathcal{P}_{\eta_2}\subset \mathcal{Q}_{\lambda_2}$. From \eqref{eqn: lambda1 upper bound}, \eqref{eqn: eta1 upper bound}, \eqref{eqn: lambda2 upper bound} and \eqref{eqn: eta2 upper bound}, it follows immediately that $\eta_2$ is given by
\begin{align}
\label{}
\eta_2 
&= \lambda_2^{\frac{1}{m}}\,\Bigg(\sum_{i=1}^{n}\frac{\alpha_i}{\sqrt[m-1]{d_i}}\Bigg)^{\frac{m-1}{m}}  
  = \eta_1\,\bigg(\max_{1\le i \le n} \frac{d_i}{\alpha_i^{m-1}}\bigg)^{\frac{1}{m}}
    \Bigg(\sum_{i=1}^{n}\frac{\alpha_i}{\sqrt[m-1]{d_i}}\Bigg)^{\frac{m-1}{m}}\\ \notag
&= \lambda_1^{\frac{1}{m}}\,\bigg(\max_{1\le i \le n} \frac{d_i}{\alpha_i^{m-1}}\bigg)^{\frac{1}{m}}      
     \Bigg(\sum_{i=1}^{n}\frac{\alpha_i}{\sqrt[m-1]{d_i}}\Bigg)^{\frac{2\,(m-1)}{m}} \\ \notag    
&=\bigg(\max_{1\le i \le n}\alpha_i\,d_i\,\bar{u}_i^m\bigg)^{\frac{1}{m}}\,\bigg(\max_{1\le i \le n} \frac{d_i}{\alpha_i^{m-1}}\bigg)^{\frac{1}{m}}      
     \Bigg(\sum_{i=1}^{n}\frac{\alpha_i}{\sqrt[m-1]{d_i}}\Bigg)^{\frac{2\,(m-1)}{m}}.     
\end{align}

(An illustration of the N-barrier for $m=n=2$ is given in Remark~\ref{rem: N-barrier for upper bounds}.)

As the proof of Proposition~\ref{prop: lower bed}, we claim by contradiction that $q(x)\le\lambda_2$ for $x\in\mathbb{R}$, from which \eqref{eqn: upper bound of p} follows since the $\alpha_i>0$ $(i=1,2,\cdots,n)$ are arbitrary and the relationship $\mathcal{P}_{\eta_2}\supset\mathcal{Q}_{\lambda_2}$ holds. Suppose that, contrary to our claim, there exists $z\in\mathbb{R}$ such that $q(z)>\lambda_2$. Since $u_i(x)\in C^2(\mathbb{R})$ and  $(u_1,u_2,\cdots,u_n)(\pm\infty)=\textbf{e}_{\pm}$, we may assume $\displaystyle\max_{x\in\mathbb{R}} q(x)=q(z)$. 
We denote respectively by $z_2$ and $z_1$ the first points at which the solution $(u_1(x),u_2(x),\cdots,u_n(x))$ intersects the hyper-ellipsoid $\displaystyle\sum_{i=1}^{n} \alpha_i\,d_i\,u_i^m=\lambda_1$ when $x$ moves from $z$ towards $\infty$ and $-\infty$. 
For the case where $\theta\leq0$, we integrate \eqref{eqn: ODE for p and q upper bound} with respect to $x$ from $z_1$ to $z$ and obtain
\begin{equation}\label{eqn: integrating eqn upper bound}
q'(z)-q'(z_1)+\theta\,(p(z)-p(z_1))+\int_{z_1}^{z}F(u_1(x),u_2(x),\cdots,u_n(x))\,dx\geq0.
\end{equation}
On the other hand we have:
\begin{itemize}
  \item $q'(z)=0$ because of $\displaystyle\max_{x\in\mathbb{R}} q(x)=q(z)$;
  \item $q(z_1)=\lambda_1$ follows from the fact that $z_1$ is on the hyper-ellipsoid $\displaystyle\sum_{i=1}^{n} \alpha_i\,d_i\,u_i^m=\lambda_1$. Since $z_1$ is the first point
  for $q(x)$ taking the value $\lambda_1$ when $x$ moves from $z$ to $-\infty$, we conclude that $q(z_1+\delta)\ge\lambda_1$ for $z-z_1>\delta>0$ and $q'(z_1)\ge 0$;
  \item $p(z)>\eta_1$ since $z$ is above the hyperplane $\displaystyle\sum_{i=1}^{n} \alpha_i\,u_i=\eta_1$; $p(z_1)<\eta_1$ since $z_1$ is below the hyperplane $\displaystyle\sum_{i=1}^{n} \alpha_i\,u_i=\eta_1$;
  \item let $\mathcal{F}_-=\Big\{(u_1,u_2,\cdots,u_n)\,\Big|\, F(u_1,u_2,\cdots,u_n)<0, u_1,u_2,\cdots,u_n\ge0\Big\}$. 
  Due to the fact that $(u_1(z_1),u_2(z_1),\cdots,u_n(z_1))$ is on the hyper-ellipsoid $\displaystyle\sum_{i=1}^{n} \alpha_i\,d_i\,u_i^m=\lambda_1$ and $(u_1(z),u_2(z),\cdots,u_n(z))$ $\in\mathcal{Q}_{\lambda_2}$, $(u_1(z_1),u_2(z_1),\cdots,u_n(z_1))$, $(u_1(z),u_2(z),\cdots,u_n(z))$ $\in\underaccent\bar{\mathcal{R}}$ by \eqref{eqn: Q1<P<Q2<R upper bound}. Because of $\mathbf{[\bar{H}]}$ and $F(u_1,u_2,\cdots,u_n)=\displaystyle\sum_{i=1}^{n} \alpha_i\,u_i^{l_i}\,f_i(u_1,u_2,\cdots,u_n)$, it is easy to see that 
\begin{equation}
\Big\{(u_1(x),u_2(x),\cdots,u_n(x))\,\Big|\,z_1\le x \le z\Big\} \subset\bar{\mathcal{R}} \subset \mathcal{F}_-.
\end{equation}  
Therefore we have $\displaystyle\int_{z_1}^{z}F(u_1(x),u_2(x),\cdots,u_n(x))\,dx<0$.
\end{itemize}

Combining the above arguments, we obtain
\begin{equation}
q'(z)-q'(z_1)+\theta\,(p(z)-p(z_1))+\int_{z_1}^{z}F(u_1(x),u_2(x),\cdots,u_n(x))\,dx<0,
\end{equation}
which contradicts \eqref{eqn: integrating eqn upper bound}. Therefore when $\theta\leq0$, $q(x)\leq \lambda_2$ for $x\in \mathbb{R}$. For the case where $\theta\geq0$, integrating \eqref{eqn: ODE for p and q upper bound} with respect to $x$ from $z$ to $z_2$ yields
\begin{equation}\label{eqn: eqn by integrate from z to z2 upper bound}
q'(z_2)-q'(z)+\theta\,(p(z_2)-p(z))+\int_{z}^{z_2}F(u_1(x),u_2(x),\cdots,u_n(x))\,dx\geq0.
\end{equation}
In a similar manner, it can be shown that $q'(z_2)\le 0$, $q'(z)=0$, $p(z_2)<\eta$, $p(z)>\eta$, and 
\begin{equation}
\int_{z}^{z_2}F(u_1(x),u_2(x),\cdots,u_n(x))\,dx<0.
\end{equation}
These together contradict \eqref{eqn: eqn by integrate from z to z2 upper bound}. Consequently, \eqref{eqn: upper bound of p} is proved and the proof is completed.

\end{proof}

\begin{remark}[\textbf{N-barrier for upper bounds}]\label{rem: N-barrier for upper bounds}

We illustrate the construction of the N-barrier in Proposition~\ref{prop: upper bed} for the case when $m=n=2$. For consistency, we use the setting in Remark~\ref{rem: N-barrier for lower bounds}.


\begin{itemize}
  \item [(i)] \textbf{\underline{Ellipse $\alpha\,d_1\,u^2+\beta\,d_2\,v^2=\lambda_1$}}\ \ We first determine the ellipse $\alpha\,d_1\,u^2+\beta\,d_2\,v^2=\lambda_1$ by letting
\begin{equation}\label{eqn: lambda1 upper bound rem}
\lambda_1=
\max
\big(
\alpha\,d_1\,\bar{u}^2,\beta\,d_2\,\bar{v}^2
\big).
\end{equation}
The $u$-coordinate of the $u$-intercept and the $v$-coordinate of the $v$-intercept of the ellipse $\alpha\,d_1\,u^2+\beta\,d_2\,v^2=\lambda_1$ are $\displaystyle\sqrt{\frac{\lambda_1}{\alpha\,d_1}}$ and $\displaystyle\sqrt{\frac{\lambda_1}{\beta\,d_2}}$, respectively; the $u$-coordinate of the $u$-intercept and the $v$-coordinate of the line $\displaystyle\frac{u}{\bar{u}}+\displaystyle\frac{v}{\bar{v}}=1$ are $\bar{u}$ and $\bar{v}$, respectively. It turns out that
\begin{itemize}
  \item when $\displaystyle\max\big(\alpha\,d_1\,\bar{u}^2,\beta\,d_2\,\bar{v}^2\big)=\alpha\,d_1\,\bar{u}^2$, we have 
\begin{alignat}{3}
\displaystyle\sqrt{\frac{\lambda_1}{\alpha\,d_1}} & = \bar{u}, & \qquad 
\displaystyle\sqrt{\frac{\lambda_1}{\beta\,d_2}} &=\bar{u}\,\displaystyle\sqrt{\frac{\alpha\,d_1}{\beta\,d_2}}\ge\bar{u}\,\displaystyle\sqrt{\frac{\bar{v}^2}{\bar{u}^2}} = \bar{v};
\end{alignat}
  \item when $\displaystyle\max\big(\alpha\,d_1\,\bar{u}^2,\beta\,d_2\,\bar{v}^2\big)=\beta\,d_2\,\bar{v}^2$, we have 
\begin{alignat}{3}
\displaystyle\sqrt{\frac{\lambda_1}{\beta\,d_2}} & = \bar{v}, & \qquad 
\displaystyle\sqrt{\frac{\lambda_1}{\alpha\,d_1}} &=\bar{v}\,\displaystyle\sqrt{\frac{\beta\,d_2}{\alpha\,d_1}}\ge\bar{v}\,\displaystyle\sqrt{\frac{\bar{u}^2}{\bar{v}^2}} = \bar{u}.
\end{alignat}
\end{itemize}
This means that the ellipse $\alpha\,d_1\,u^2+\beta\,d_2\,v^2=\lambda_1$ lies \textit{entirely above} the line $\displaystyle\frac{u}{\bar{u}}+\displaystyle\frac{v}{\bar{v}}=1$  in the first quadrant of the $uv$-plane. 
  \item [(ii)] \textbf{\underline{Line $\eta_1=\alpha\,u+\beta\,v$}}\ \ Since the line $\eta_1=\alpha\,u+\beta\,v$ is tangent to the ellipse $\alpha\,d_1\,u^2+\beta\,d_2\,v^2=\lambda_1$ at $(u,v)$ in the first quadrant of the $uv$-plane, we have the following equations:
\begin{eqnarray}
\frac{\alpha\,d_1\,u}{\beta\,d_2\,v} & = & \frac{\alpha}{\beta},\\
\alpha\,u+\beta\,v & = & \eta_1,\\
\alpha\,d_1\,u^2+\beta\,d_2\,v^2 & = & \lambda_1.
\end{eqnarray}
Employing Lemma~\ref{lem: lambda2} in Section~\ref{sec: appendix}, we obtain
\begin{equation}\label{eqn: eta1 upper bound rem}
\eta_1=\sqrt{\lambda_1\,\bigg(\frac{\alpha}{d_1}+\frac{\beta}{d_2}\bigg)}.
\end{equation}
We note that the line $\eta_1=\alpha\,u+\beta\,v$  lies \textit{entirely above} the ellipse $\alpha\,d_1\,u^2+\beta\,d_2\,v^2=\lambda_1$   in the first quadrant of the $uv$-plane.
  \item [(iii)] \textbf{\underline{Ellipse $\alpha\,d_1\,u^2+\beta\,d_2\,v^2=\lambda_2$}}\ \ We determine the ellipse $\alpha\,d_1\,u^2+\beta\,d_2\,v^2=\lambda_2$ by letting
\begin{equation}\label{eqn: lambda2 upper bound rem}
\lambda_2=
\eta_1^2\,
\max
\bigg(
\frac{d_1}{\alpha},\frac{d_2}{\beta}
\bigg).
\end{equation}
The $u$-coordinate of the $u$-intercept and the $v$-coordinate of the $v$-intercept of the ellipse $\alpha\,d_1\,u^2+\beta\,d_2\,v^2=\lambda_2$ are $\displaystyle\sqrt{\frac{\lambda_2}{\alpha\,d_1}}$ and $\displaystyle\sqrt{\frac{\lambda_2}{\beta\,d_2}}$, respectively; the $u$-coordinate of the $u$-intercept and the $v$-coordinate of the line $\eta_1=\alpha\,u+\beta\,v$ are $\displaystyle\frac{\eta_1}{\alpha}$ and $\displaystyle\frac{\eta_1}{\beta}$, respectively. It follows that
\begin{itemize}
  \item when $\displaystyle\max\bigg(\frac{d_1}{\alpha},\frac{d_2}{\beta}\bigg)=\frac{d_1}{\alpha}$, we have 
\begin{alignat}{3}
\displaystyle\sqrt{\frac{\lambda_2}{\alpha\,d_1}} & = \frac{\eta_1}{\alpha}, & \qquad 
\displaystyle\sqrt{\frac{\lambda_2}{\beta\,d_2}} &=\eta_1\,\sqrt{\frac{d_1}{\alpha\,\beta\,d_2}}\ge \eta_1\,\sqrt{\frac{\alpha\,d_2}{\alpha\,\beta^2\,d_2}}=\frac{\eta_1}{\beta};
\end{alignat}
  \item when $\displaystyle\max\bigg(\frac{d_1}{\alpha},\frac{d_2}{\beta}\bigg)=\frac{d_2}{\beta}$, we have 
\begin{alignat}{3}
\displaystyle\sqrt{\frac{\lambda_2}{\beta\,d_2}} & = \frac{\eta_2}{\beta}, & \qquad 
\displaystyle\sqrt{\frac{\lambda_2}{\alpha\,d_1}} &=\eta_1\,\sqrt{\frac{d_2}{\beta\,\alpha\,d_1}}\ge \eta_1\,\sqrt{\frac{\beta\,d_1}{\alpha^2\,\beta\,d_1}}=\frac{\eta_1}{\alpha}.
\end{alignat}
\end{itemize}
We see from the construction of the ellipse $\alpha\,d_1\,u^2+\beta\,d_2\,v^2=\lambda_2$ that the ellipse $\alpha\,d_1\,u^2+\beta\,d_2\,v^2=\lambda_2$ lies \textit{entirely above} the line $\eta_1=\alpha\,u+\beta\,v$ in the first quadrant of the $uv$-plane. 
\end{itemize}

The two ellipses $\alpha\,d_1\,u^2+\beta\,d_2\,v^2=\lambda_1$ and $\alpha\,d_1\,u^2+\beta\,d_2\,v^2=\lambda_2$, and the line $\eta_1=\alpha\,u+\beta\,v$ form the N-barrier. Finally, we find the tangent line of the ellipse $\alpha\,d_1\,u^2+\beta\,d_2\,v^2=\lambda_2$ in the first quadrant of the $uv$-plane by determining the line $\eta_2=\alpha\,u+\beta\,v$ as in step (ii):
\begin{eqnarray}
\frac{\alpha\,d_1\,u}{\beta\,d_2\,v} & = & \frac{\alpha}{\beta},\\
\alpha\,u+\beta\,v & = & \eta_2,\\
\alpha\,d_1\,u^2+\beta\,d_2\,v^2 & = & \lambda_2.
\end{eqnarray}
We obtain
\begin{equation}\label{eqn: eta2 upper bound rem}
\eta_2=\sqrt{\lambda_2\,\bigg(\frac{\alpha}{d_1}+\frac{\beta}{d_2}\bigg)}
\end{equation}
or
\begin{equation}
\eta_2 = 
\bigg(
\frac{\alpha}{d_1}+\frac{\beta}{d_2}
\bigg) 
\sqrt{
\max
\bigg(
\frac{d_1}{\alpha},\frac{d_2}{\beta}
\bigg)
\max
\Big(
\alpha\,d_1\,\bar{u}^2,\beta\,d_2\,\bar{v}^2
\Big)
}
\end{equation}
by combining \eqref{eqn: lambda1 upper bound rem}, \eqref{eqn: eta1 upper bound rem}, \eqref{eqn: lambda2 upper bound rem} and \eqref{eqn: eta2 upper bound rem}.

It is readily seen from that, depending on $\displaystyle\max\big(\alpha\,d_1\,\bar{u}^2,\beta\,d_2\,\bar{v}^2\big)$ and $\displaystyle\max\bigg(\frac{d_1}{\alpha},\frac{d_2}{\beta}\bigg)$,

\begin{figure}
\centering
\mbox{
\subfigure[$\displaystyle\frac{d_1}{\alpha}>\frac{d_2}{\beta}$, $\alpha\,d_1\,\bar{u}^2<\beta\,d_2\,\bar{v}^2$: $d_2=4$, $\beta=2$, $\lambda_1=8$, $\lambda_2=20$, $\eta_1=\displaystyle2\,\sqrt{\frac{5}{3}}$, $\eta_2=\displaystyle5\,\sqrt{\frac{2}{3}}$.]{\includegraphics[width=0.48\textwidth]{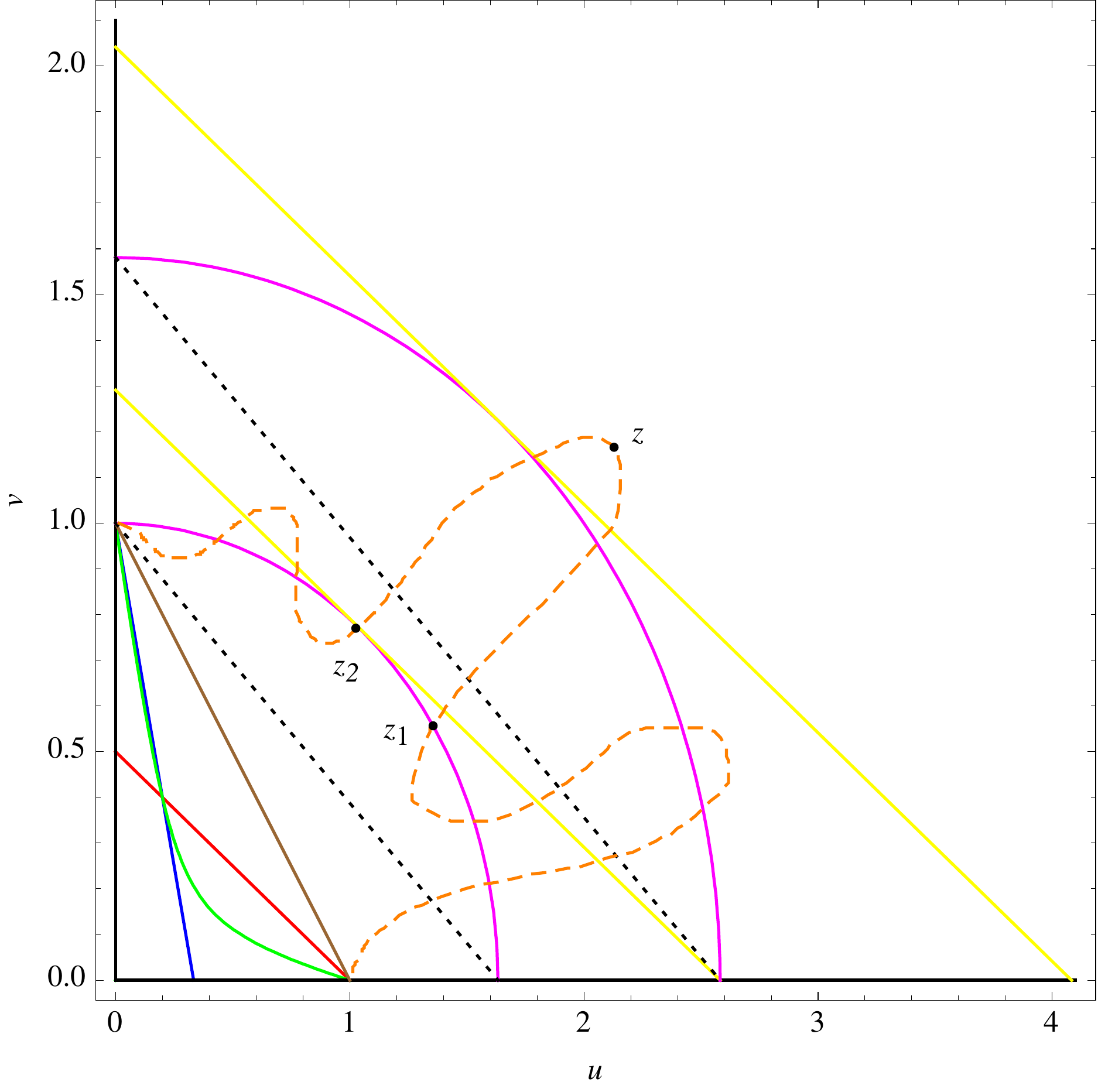}
             \label{fig: upper bound d1/alpha-beta d2 v2}    } \quad \hspace{0mm}
\subfigure[$\displaystyle\frac{d_1}{\alpha}<\frac{d_2}{\beta}$, $\alpha\,d_1\,\bar{u}^2<\beta\,d_2\,\bar{v}^2$: $d_2=4$, $\beta=1$, $\lambda_1=4$, $\lambda_2=\displaystyle\frac{28}{3}$, $\eta_1=\displaystyle\sqrt{\frac{7}{3}}$, $\eta_2=\displaystyle\frac{7}{3}$.]{\includegraphics[width=0.48\textwidth]{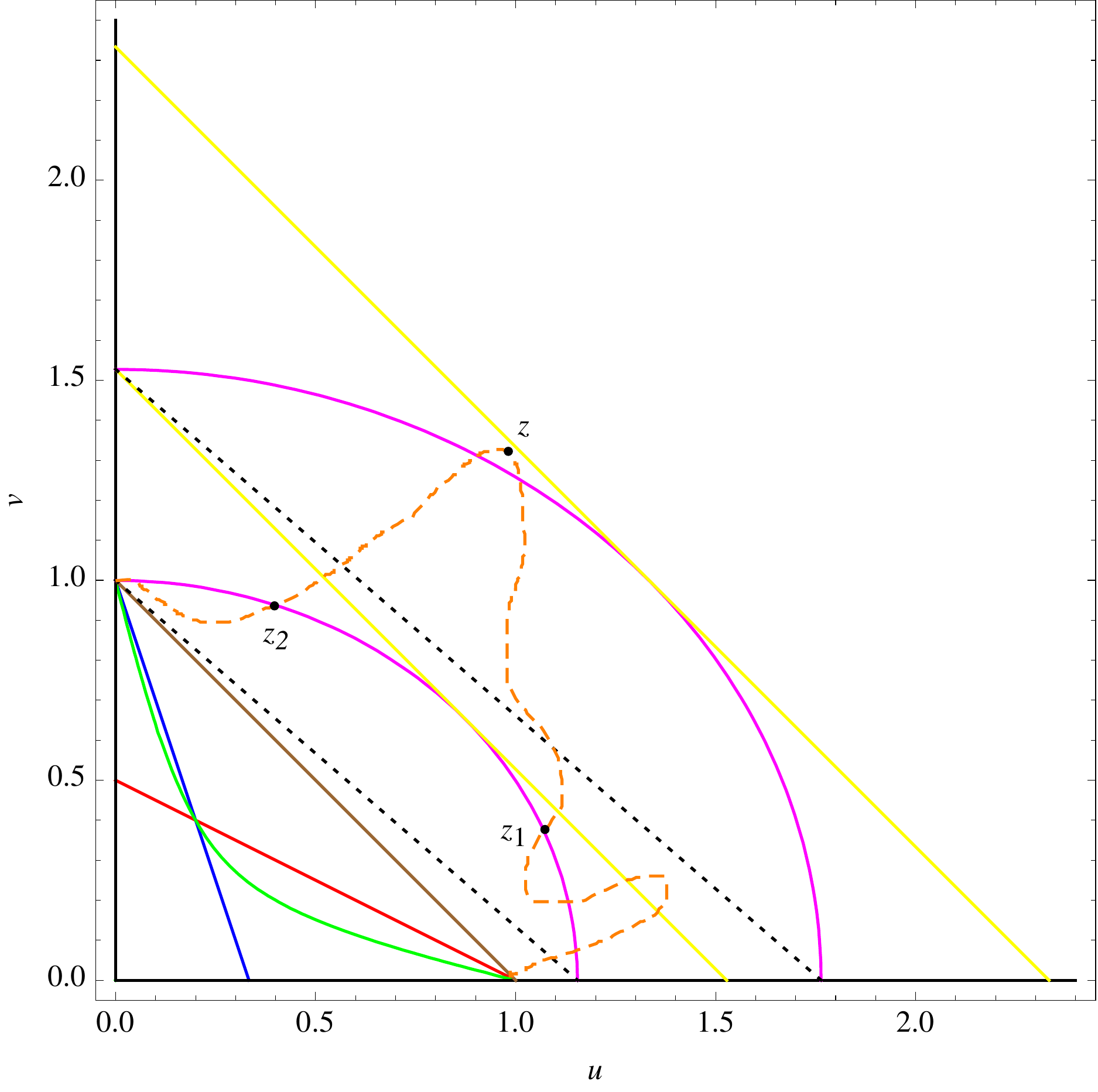}
             \label{fig: upper bound d2/beta-beta d2 v2}    } \quad \hspace{0mm}
             }
\mbox{
\subfigure[$\displaystyle\frac{d_1}{\alpha}<\frac{d_2}{\beta}$, $\alpha\,d_1\,\bar{u}^2>\beta\,d_2\,\bar{v}^2$: $d_2=4$, $\beta=\displaystyle\frac{1}{2}$, $\lambda_1=3$, $\lambda_2=11$, $\eta_1=\displaystyle\frac{1}{2}\,\sqrt{\frac{11}{2}}$, $\eta_2=\displaystyle\frac{11}{2\,\sqrt{6}}$.]{\includegraphics[width=0.48\textwidth]{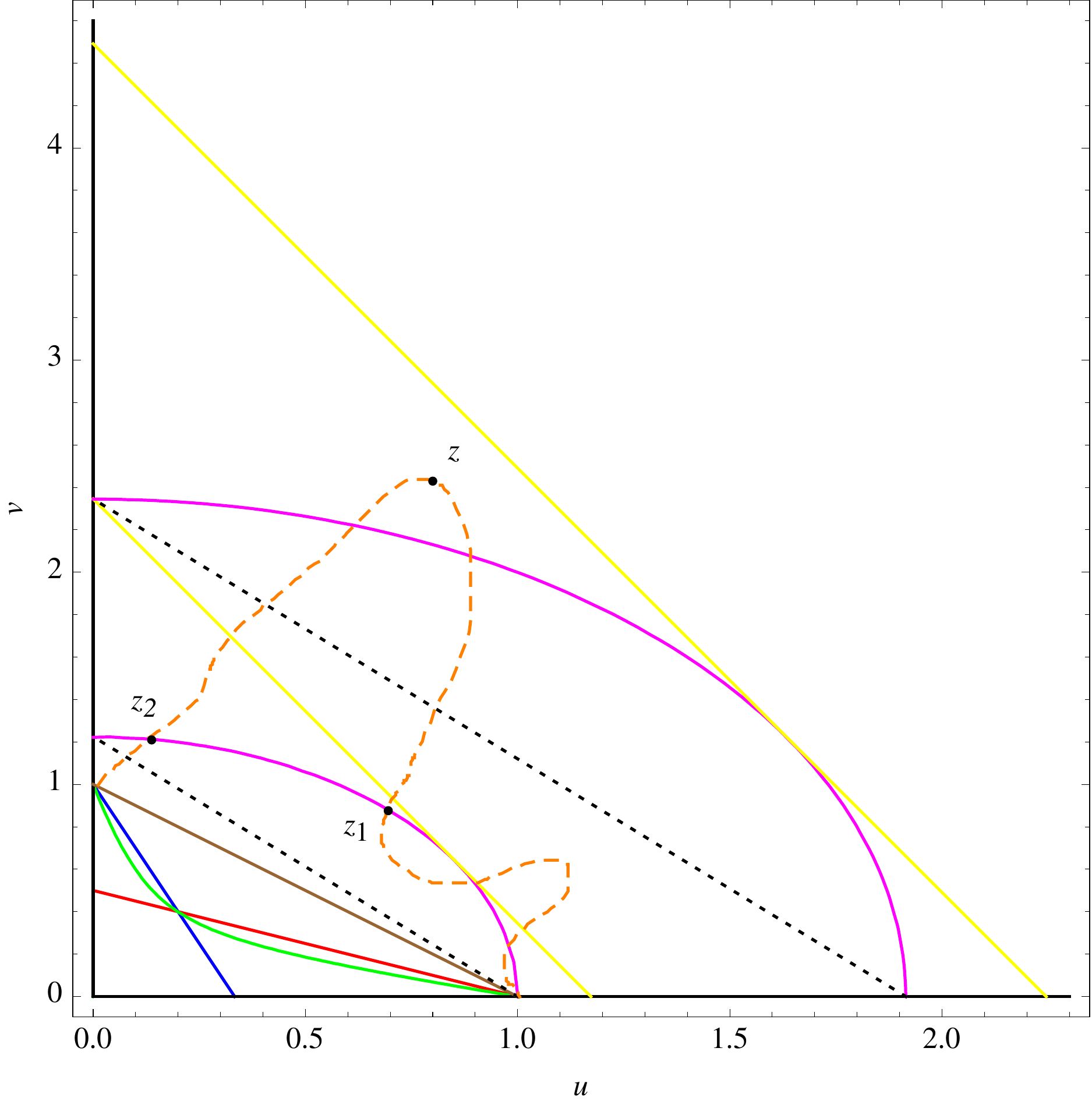}
             \label{fig: upper bound d2/beta-alpha d1 u2}    } \quad \hspace{0mm}
\subfigure[$\displaystyle\frac{d_1}{\alpha}>\frac{d_2}{\beta}$, $\alpha\,d_1\,\bar{u}^2>\beta\,d_2\,\bar{v}^2$: $d_2=2$, $\beta=\displaystyle\frac{3}{4}$, $\lambda_1=3$, $\lambda_2=\displaystyle\frac{51}{8}$, $\eta_1=\displaystyle\frac{1}{2}\,\sqrt{\frac{17}{2}}$, $\eta_2=\displaystyle\frac{17}{8}$.]{\includegraphics[width=0.48\textwidth]{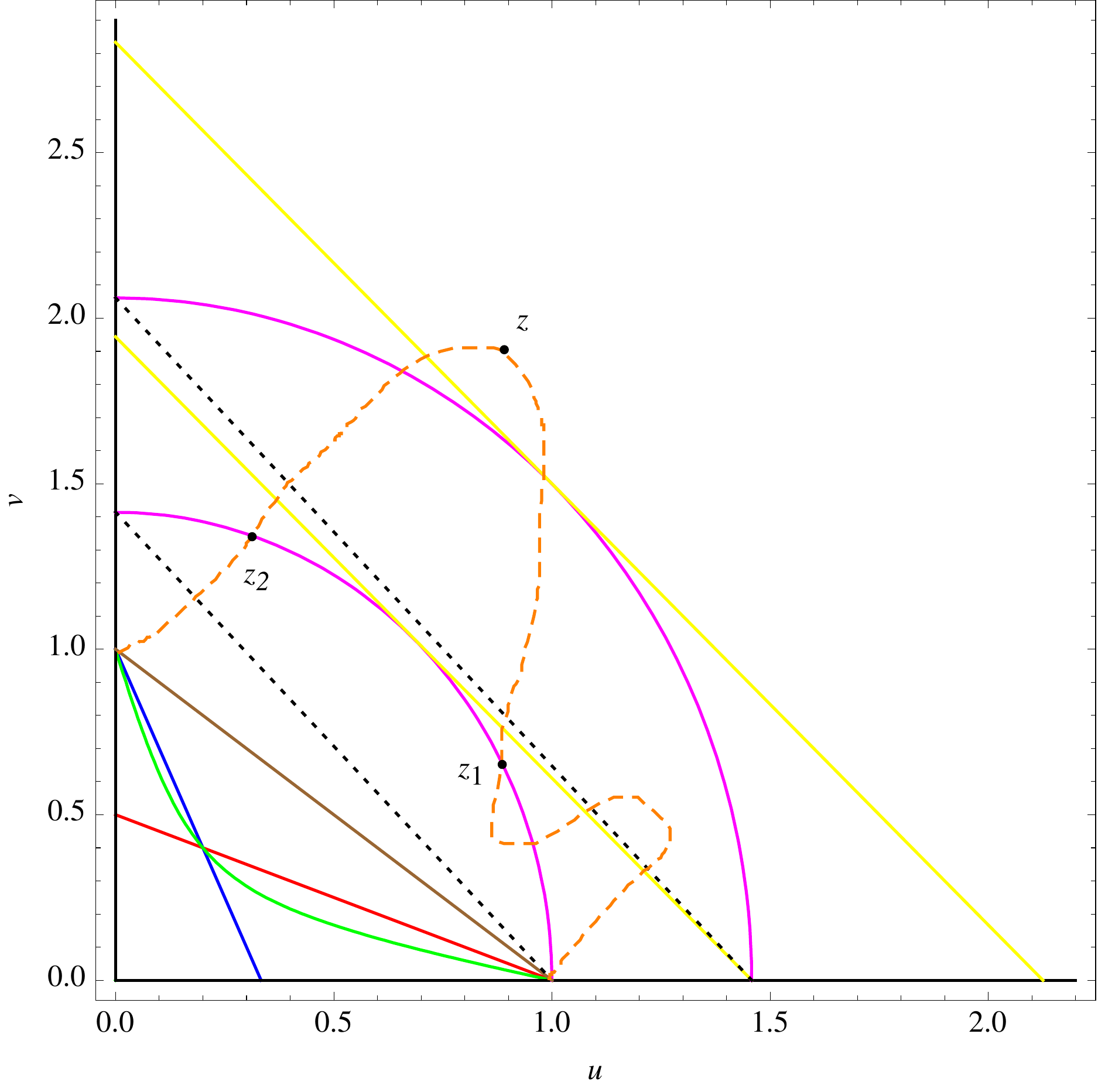}
             \label{fig: upper bound d1/alpha-alpha d1 u2}    } \quad \hspace{0mm}
             }
\caption{\small Red line: $1-u-a_1\,v=0$; blue line: $1-a_2\,u-v=0$; green curve: $F(u,v)=0$; brown line: $\displaystyle\frac{u}{\overline{u}}+\frac{v}{\overline{v}}=1$, where $\overline{u}$ and $\overline{v}$ are given by \eqref{eqn: ubar rem} and \eqref{eqn: vbar rem}; magenta ellipses : $\alpha\,d_1\,u^2+\beta\,d_2\,v^2=\lambda_1,\lambda_2$, where $\lambda_1$ (below) is given by \eqref{eqn: lambda1 upper bound rem} and $\lambda_2$ (above) by \eqref{eqn: lambda2 upper bound rem}; yellow lines: $\alpha\,u+\beta\,v=\eta_1,\eta_2$, where $\eta_1$ (below) is given by \eqref{eqn: eta1 upper bound rem} and $\eta_2$ (above) by \eqref{eqn: eta2 upper bound rem}; dashed orange curve: the solution $(u(x),v(x))$; dotted lines: $\displaystyle\sqrt{\alpha\,d_1}\,u+\sqrt{\beta\,d_2}\,v=\sqrt{\lambda_1}$ (below), $\displaystyle\sqrt{\lambda_2}$ (above); $\overline{u}=\overline{v}=1$; $d_1=3$, $a_1=2$, $a_2=3$, $\alpha=1$.
\label{fig: N-barrier upper bound}}
\end{figure}

\end{remark}

We are now in the position to prove Theorem~\ref{thm: NBMP for n species}.

\begin{proof}[Proof of Theorem~\ref{thm: NBMP for n species}]
In Propositions~\ref{prop: lower bed} and \ref{prop: upper bed}, we obtain a lower and upper bound for $\displaystyle\sum_{i=1}^{n} \alpha_i\,u_i(x)$, respectively. Combining the results in Propositions~\ref{prop: lower bed} and \ref{prop: upper bed}, we immediately establish Theorem~\ref{thm: NBMP for n species}.
\end{proof}

\begin{remark}
\textit{The tanh method} \cite{hung2012JJIAM,CHMU,Rodrigo&Mimura00,Rodrigo&Mimura01} allows us to find exact solutions to \textbf{(BVP)} with certain class of the nonlinearity. For instance, when $m=n=2$, \textbf{(BVP)} with Zeldovich-type reaction terms (\cite{Zeldovich-51-flame-propagation,Gilding-Kersner-TW-04-diffusion-convection reaction,Barenblatt&Zeldovich80MathThy-Combustion-Explosion}) becomes
\begin{equation}\label{eqn:  2 species exact TWS example}
\begin{cases}
\vspace{3mm} 
d_1(u^2)_{xx}+\theta\,u_x+u^2\,(\sigma_1-c_{11}\,u-c_{12}\,v)=0,\ \ &x\in\mathbb{R},\\
\vspace{3mm} 
d_2(v^2)_{xx}+\theta\,v_x+v^2\,(\sigma_2-c_{21}\,u-c_{22}\,v)=0,\ \ &x\in\mathbb{R},\\
(u,v)(-\infty)=\Big(\displaystyle\frac{\sigma_1}{c_{11}},0\Big),\ \ (u,v)(+\infty)= \Big(0,\displaystyle\frac{\sigma_2}{c_{22}}\Big).
\end{cases}
\end{equation}
Applying Theorem~\ref{thm: exact TWS appendix} (see Appendix~\ref{subsec: exact solutions using Tanh method}), we see that when $c_{11}=1$, $c_{22}=2$, $d_1=3$, and $d_2=4$, \eqref{eqn: soln of algebraic eqn tanh method} gives $\theta=0$, $k_1=60$, $k_2=8$, $\sigma_1=240$, $\sigma_2=32$, $c_{12}=27$, and $c_{21}=\displaystyle\frac{2}{5}$, and hence \eqref{eqn:  2 species exact TWS example} admits the solution (see Figure~\ref{fig: exact TW Zeldovich-type})
\begin{equation}
\begin{cases}
\vspace{3mm}
u(x) = 60\,\big(1-\tanh x\big)^2, \ \ &x\in\mathbb{R},\\
v(x) = 8\,\big(1+\tanh x\big), \ \ &x\in\mathbb{R}.
\end{cases}
\end{equation}
Letting $\alpha=\displaystyle\frac{1}{2}$ and $\beta=\displaystyle\frac{1}{3}$, it follows immediately that $\displaystyle\alpha\,u(x)+\beta\,v(x)=30\,\tanh^2 x-\frac{172}{3}\,\tanh x+\frac{98}{3}$ is monotonically decreasing in $x$. As a result, 
\begin{equation}
\displaystyle\frac{16}{3}=\alpha\,u(\infty)+\beta\,v(\infty)\le\alpha\,u(x)+\beta\,v(x)\le\alpha\,u(-\infty)+\beta\,v(-\infty)=120, \ \ x\in\mathbb{R}.
\end{equation}
On the other hand, upper and lower bounds given by Corollary~\ref{cor: NBMP for NDC-tw} turn out to be 
\begin{equation}
1.70\approx\frac{80}{\sqrt{2211}}=\underaccent\bar{\lambda}\le\alpha\,u(x)+\beta\,v(x)\le\bar{\lambda}=180\,\sqrt{2}\approx254.56,  \ \ x\in\mathbb{R}.
\end{equation}
Thus, we verify Corollary~\ref{cor: NBMP for NDC-tw} in this case.
 $\underaccent\bar{u}=80$, $\underaccent\bar{v}=\displaystyle\frac{80}{9}$, $\bar{u}=240$, $\bar{v}=16$,

\begin{figure}[ht!]
\centering
\mbox{
{\includegraphics[width=0.50\textwidth]{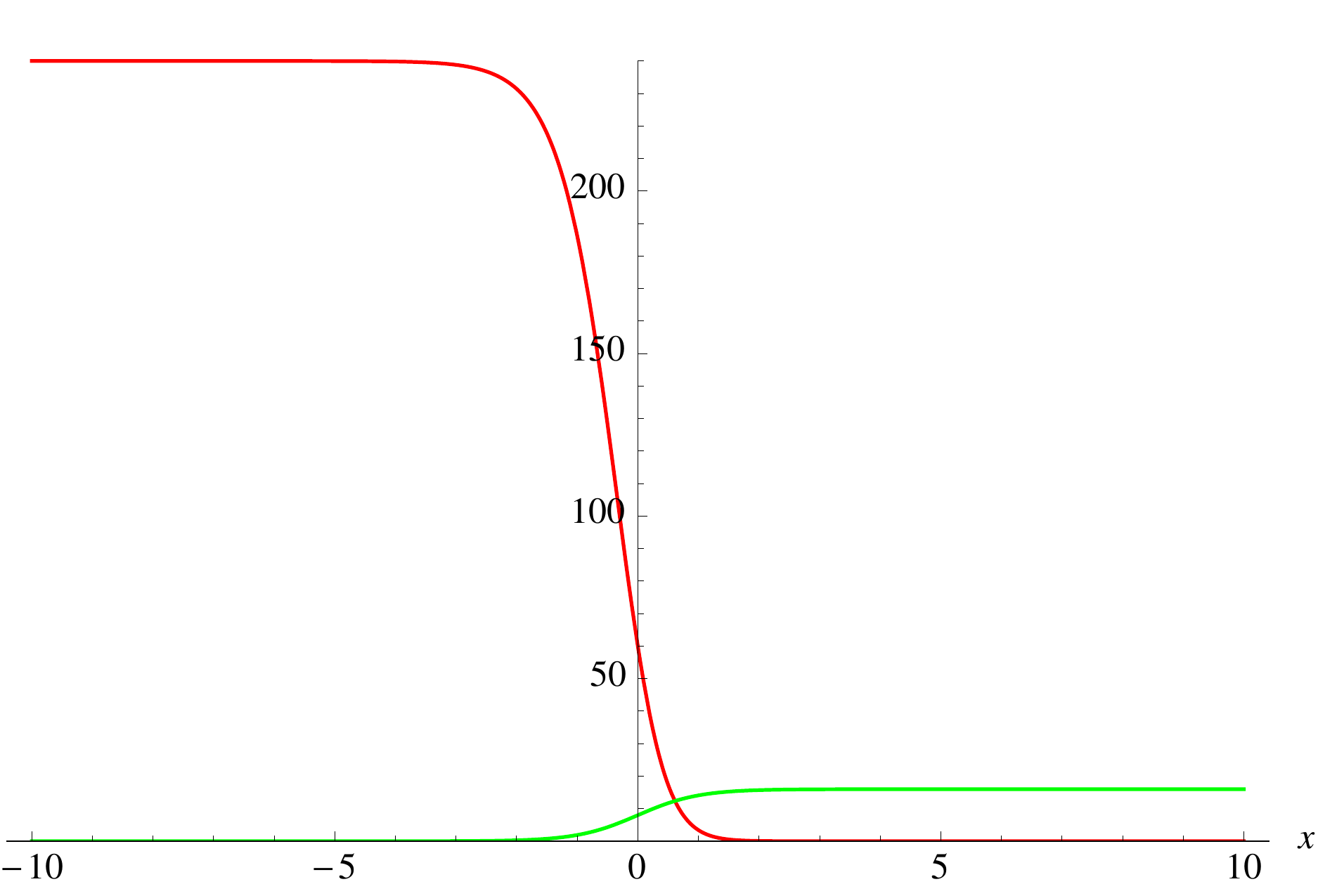}
             } \quad \hspace{0mm}
             }
\caption{\small Red: $u(x) =60\,\big(1-\tanh x\big)^2$; green: $v(x) =8\,\big(1+\tanh x\big)$. 
\label{fig: exact TW Zeldovich-type}}
\end{figure}

\end{remark}

\section{Application to the nonexistence of three species traveling waves: proof of Theorem~\ref{thm: Nonexistence 3 species} }\label{sec: nonexistence}


In this section, we prove Theorem~\ref{thm: Nonexistence 3 species} by contradiction.
\begin{proof}[Proof of Theorem~\ref{thm: Nonexistence 3 species}]

We first prove \textbf{\textit{(i)}}. Suppose to the contrary that there exists a solution $(u(x),v(x),w(x))$ to the problem \eqref{eqn: 3 species TWS nonexistence}. Due to $\mathbf{[H1]}$, we have $w_x(x_0)=0$ and $w_{xx}(x_0)\le 0$. Since $w(x)$ satisfies
\begin{equation}
d_3\,(w^2)_{xx}+\theta\,w_x+w\,(\sigma_3-c_{31}\,u-c_{32}\,v-c_{33}\,w)=0
\end{equation}
and $(w^2)_{xx}=2\,(w_x^2+w\,w_{xx})$, we obtain
\begin{equation}\label{eqn: w(???) ???>0}
\sigma_3-c_{31}\,u(x_0)-c_{32}\,v(x_0)-c_{33}\,w(x_0)\ge 0.
\end{equation}
This lead to an upper bound of $w(x)$, i.e.
\begin{equation}
w(x)\leq w(x_0)\le \displaystyle\frac{1}{c_{33}}\big(\sigma_3-c_{31}\,u(x_0)-c_{32}\,v(x_0)\big)<\displaystyle\frac{\sigma_3}{c_{33}},\ \ x\in\mathbb{R}.
\end{equation}
By virtue of the inequality $w(x)<\displaystyle\frac{\sigma_3}{c_{33}}$, the last two equations in \eqref{eqn: 3 species TWS nonexistence} become
\begin{equation}\label{eqn: nonexistence diff ineq <0}
\begin{cases}
\vspace{3mm} 
d_1\,(u^2)_{xx}+\theta \,u_x+u\,(\sigma_1-c_{13}\,\sigma_3\,c_{33}^{-1}-c_{11}\,u-c_{12}\,v)\leq0,\ \ &x\in\mathbb{R},\\
d_2\,(v^2)_{xx}+\theta \,v_x+v\,(\sigma_2-c_{23}\,\sigma_3\,c_{33}^{-1}-c_{21}\,u-c_{22}\,v)\leq0, \ \ &x\in\mathbb{R}.
\end{cases}
\end{equation}
By means of $\mathbf{[H0]}$ and $\mathbf{[H2]}$, we can employ Corollary~\ref{cor: NBMP for NDC-tw} with $u_1=u$, $u_2=v$ and $\alpha_1=c_{31}$, $\alpha_2=c_{32}$ to obtain a lower bound of $c_{31}\,u(x)+c_{32}\,v(x)$, i.e.
\begin{equation}
c_{31}\,u(x)+c_{32}\,v(x)\geq d_1\,d_2\,\underaccent\bar{u}_{\ast}\,\underaccent\bar{v}_{\ast}\,\displaystyle\min \Big(\frac{c_{31}}{d_1},\frac{c_{32}}{d_2}\Big)\,\displaystyle\sqrt{\frac{\displaystyle c_{31}\,c_{32}}{(\displaystyle c_{31}\,d_1\,\underaccent\bar{u}_{\ast}^2+c_{32}\,d_2\,\underaccent\bar{v}_{\ast}^2)\,(c_{31}\,d_2+c_{32}\,d_1)}},\ \ x\in\mathbb{R}.
\end{equation}
However, $\mathbf{[H3]}$ yields
\begin{equation}
c_{31}\,u(x)+c_{32}\,v(x)\geq\lambda_{\ast}\geq\sigma_3,\;x\in\mathbb{R},
\end{equation}
which contradicts \eqref{eqn: w(???) ???>0}. This completes the proof of \textbf{\textit{(i)}}. To prove \textit{\textbf{(ii)}}, an easy observation leads to 
\begin{equation}
\begin{cases}
\vspace{3mm} 
d_1\,(u^2)_{xx}+\theta \,u_x+u\,(\sigma_1-c_{11}\,u-c_{12}\,v)>0,\ \ &x\in\mathbb{R},\\
d_2\,(v^2)_{xx}+\theta \,v_x+v\,(\sigma_2-c_{21}\,u-c_{22}\,v)>0, \ \ &x\in\mathbb{R},
\end{cases}
\end{equation}
since $w(x)>0$, $x\in\mathbb{R}$. Letting $u_1=u$, $u_2=v$ and $\alpha_1=c_{31}$, $\alpha_2=c_{32}$, an upper bound of $c_{31}\,u(x)+c_{32}\,v(x)$ given by Corollary~\ref{cor: NBMP for NDC-tw} is
\begin{equation}
c_{31}\,u(x)+c_{32}\,v(x)\leq \displaystyle\bigg(
\frac{c_{31}}{d_1}+\frac{c_{32}}{d_2}
\bigg) 
\sqrt{
\max
\bigg(
\frac{d_1}{c_{31}},\frac{d_2}{c_{32}}
\bigg)
\max
\Big(
c_{31}\,d_1\,{\bar{u}^{\ast}}^2,c_{32}\,d_2\,{\bar{v}^{\ast}}^2
\Big)
}:=\lambda^{\ast},\ \ x\in\mathbb{R},
\end{equation}
where ${\bar{u}^{\ast}}$ and ${\bar{v}^{\ast}}$ are defined in $\mathbf{[H5]}$. It follows from the last inequality that
\begin{align}
\label{}
0&= d_3\,(w^2)_{xx}+\theta\,w_x+w\,(\sigma_3-c_{31}\,u-c_{32}\,v-c_{33}\,w)   \\ \notag
  &\ge d_3\,(w^2)_{xx}+\theta\,w_x+w\,(\sigma_3-\lambda^{\ast}-c_{33}\,w).
\end{align}
On the other hand, $\mathbf{[H4]}$ leads to the fact that $w_x(x_0)=0$ and $w_{xx}(x_0)\ge 0$, and hence
\begin{equation}
\sigma_3-\lambda^{\ast}-c_{33}\,w(x_0)\le 0.
\end{equation}
or
\begin{equation}
w(x)\geq w(x_0)\ge \displaystyle\frac{1}{c_{33}}\big(\sigma_3-\lambda^{\ast}\big),\ \ x\in\mathbb{R}.
\end{equation}
However, this is a contradiction with $\mathbf{[H6]}$. We complete the proof of \textbf{\textit{(ii)}}.

\end{proof}


\section{Concluding Remarks}\label{sec: Concluding Remarks}

In this paper, we have shown the NBMP for \textbf{(BVP)} with $m>1$, and apply it the establish the nonexistence of three species waves in \eqref{eqn: 3 species TWS nonexistence} under certain conditions. In particular, the upper and lower bounds given by the NBMP are verified by using exact solutions.  
 
The N-barrier method is still under investigation, and there is a number of open problems concerning NBMP. We point out some of them for further study:
\begin{itemize}
  \item \textit{NBMP for periodic solutions}:  As we can see from \cite{Guedda-Kersner-Klincsik-Logak-14-Exact-fronts-periodic}, \textbf{(NDC)} admits periodic stationary solutions under certain conditions on the parameters. Motivated by this work, we show in Theorem~\ref{thm: exact TWS appendix Cosine soln} (see Section~\ref{subsec: exact solutions of SKT}) that for the three-specie case \eqref{eqn: 3 species TWS nonexistence} also admits periodic solutions under certain conditions on the parameters. The question is how to correct the N-barrier method adapted for periodic solutions?  
  \item \textit{NBMP for multi-dimensional equations}: The N-barrier method has not yet been applied to multi-dimensional equations since there is still a lack of systematic formulation of the method in the multi-dimensional case. The difficulty is to construct appropriate N-barriers corresponding to operator like $\Delta u$, $\nabla u$, $\Delta (u^2)$ etc..  
  \item \textit{NBMP for strongly-coupled equations}: The N-barrier method developed to study \eqref{eqn: degenerate autonomous system of n species} can also be applied to a wide class of elliptic systems, for instance, the system \textbf{(SKT-tw)} in which diffusion, self-diffusion, and cross-diffusion are strongly coupled. 
\end{itemize}
These are left as the future work.

\section{Appendix}\label{sec: appendix}

\subsection{Algebraic solutions}\label{subsec: algebraic solutions}

\begin{lemma}\label{lem: lambda1}
For $\Theta$, $\Lambda>0$, if 
\begin{eqnarray}
\label{eqn: parallel}
\alpha_i\,d_i\,\underaccent\bar{u}_i\,u_i^{m-1} & = & \alpha_j\,d_j\,\underaccent\bar{u}_j\,u_j^{m-1},\ \ i,j= 1,2,\cdots,n;\\ 
\label{eqn: line}
\sum_{i=1}^{n}\frac{\displaystyle u_i}{\displaystyle\underaccent\bar{u}_i} & = & \Theta;\\
\label{eqn: ellipse}
\sum_{i=1}^{n} \alpha_i\,d_i\,u_i^m & = & \Lambda,
\end{eqnarray}
we have 
\begin{equation}
\Lambda=\Theta^m\Bigg(\sum_{i=1}^{n}\frac{1}{\sqrt[m-1]{\alpha_i\,d_i\,\underaccent\bar{u}_i^m}}\Bigg)^{1-m}.
\end{equation}
\end{lemma}

\begin{proof}
Due to \eqref{eqn: parallel}, we may assume 
\begin{equation}
u_i=\sqrt[m-1]{\frac{\displaystyle\prod_{j=1}^{n} \alpha_j\,d_j\,\underaccent\bar{u}_j}{\alpha_i\,d_i\,\underaccent\bar{u}_i}}\,K, \ \ i= 1,2,\cdots,n
\end{equation}
for some $K>0$. 
It follows immediately from \eqref{eqn: line} that $K$ is determined by
\begin{equation}
K=\frac{\Theta}{\displaystyle\sum_{i=1}^{n}\Bigg(\frac{\displaystyle1}{\displaystyle\underaccent\bar{u}_i}\sqrt[m-1]{\frac{\prod_{j=1}^{n} \alpha_j\,d_j\,\underaccent\bar{u}_j}{\alpha_i\,d_i\,\underaccent\bar{u}_i}}\Bigg)}, 
\end{equation}
and hence
\begin{align}
\label{}
u_i&=   \sqrt[m-1]{\frac{\displaystyle\prod_{j=1}^{n} \alpha_j\,d_j\,\underaccent\bar{u}_j}{\alpha_i\,d_i\,\underaccent\bar{u}_i}}\,\frac{\Theta}{\displaystyle\sum_{i=1}^{n}\Bigg(\frac{\displaystyle1}{\displaystyle\underaccent\bar{u}_i}\,\sqrt[m-1]{\frac{\prod_{j=1}^{n} \alpha_j\,d_j\,\underaccent\bar{u}_j}{\alpha_i\,d_i\,\underaccent\bar{u}_i}}\Bigg)} \\\notag
    &=  \displaystyle\frac{1}{\sqrt[m-1]{\alpha_i\,d_i\,\underaccent\bar{u}_i}}\frac{\Theta}{\displaystyle\sum_{i=1}^{n}\bigg(\frac{\displaystyle1}{\displaystyle\underaccent\bar{u}_i}\frac{1}{\sqrt[m-1]{\alpha_i\,d_i\,\underaccent\bar{u}_i}}\bigg)}\\\notag
    &=  \frac{\displaystyle \Theta}{\displaystyle\sum_{i=1}^{n}\bigg(\frac{1}{\sqrt[m-1]{\alpha_i\,d_i\,\underaccent\bar{u}_i^m}}\bigg)}\frac{1}{\sqrt[m-1]{\alpha_i\,d_i\,\underaccent\bar{u}_i}}.
\end{align}
Therefore, $\Lambda$ is given by
\begin{align*}
 \Lambda &= \sum_{i=1}^{n} \alpha_i\,d_i\,u_i^m
 =\sum_{i=1}^{n} \alpha_i\,d_i\,
    \vast(  \frac{\displaystyle \Theta}{\displaystyle\sum_{i=1}^{n}\bigg(\frac{1}{\sqrt[m-1]{\alpha_i\,d_i\,\underaccent\bar{u}_i^m}}\bigg)}\frac{1}{\sqrt[m-1]{\alpha_i\,d_i\,\underaccent\bar{u}_i}}  \vast)^m \\ \notag
    &=\Theta^m\Bigg(\sum_{i=1}^{n}\frac{1}{\sqrt[m-1]{\alpha_i\,d_i\,\underaccent\bar{u}_i^m}}\Bigg)^{-m}
    \Bigg( \sum_{i=1}^{n} \frac{\alpha_i\,d_i}{\sqrt[m-1]{(\alpha_i\,d_i\,\underaccent\bar{u}_i)^m}}\Bigg)\\ \notag
    &=\Theta^m\Bigg(\sum_{i=1}^{n}\frac{1}{\sqrt[m-1]{\alpha_i\,d_i\,\underaccent\bar{u}_i^m}}\Bigg)^{-m}
    \Bigg( \sum_{i=1}^{n} \frac{\sqrt[m-1]{(\alpha_i\,d_i)^{m-1}}}{\sqrt[m-1]{(\alpha_i\,d_i)^{m}(\underaccent\bar{u}_i)^m}}\Bigg)\\    \notag
    &=\Theta^m\Bigg(\sum_{i=1}^{n}\frac{1}{\sqrt[m-1]{\alpha_i\,d_i\,\underaccent\bar{u}_i^m}}\Bigg)^{-m}
    \Bigg( \sum_{i=1}^{n} \frac{1}{\sqrt[m-1]{\alpha_i\,d_i\,\underaccent\bar{u}_i^m}}\Bigg)=\Theta^m\Bigg(\sum_{i=1}^{n}\frac{1}{\sqrt[m-1]{\alpha_i\,d_i\,\underaccent\bar{u}_i^m}}\Bigg)^{1-m}.\\  \notag           
\end{align*}

\end{proof}

\begin{lemma}\label{lem: lambda2}
For $\Theta$, $\Lambda>0$, if 
\begin{eqnarray}
\label{eqn: parallel 2}
d_i\,u_i^{m-1} & = & d_j\,u_j^{m-1} ,\ \ i,j= 1,2,\cdots,n;\\ 
\label{eqn: line 2}
\sum_{i=1}^{n} \alpha_i\,u_i & = & \Theta;\\
\label{eqn: ellipse 2}
\sum_{i=1}^{n} \alpha_i\,d_i\,u_i^m & = & \Lambda,
\end{eqnarray}
we have 
\begin{equation}
\Lambda=\Theta^m\Bigg(\sum_{i=1}^{n}\frac{\alpha_i}{\sqrt[m-1]{d_i}}\Bigg)^{1-m}.
\end{equation}
\end{lemma}

\begin{proof}
Lemma~\ref{lem: lambda2} follows from letting $\displaystyle\underaccent\bar{u}_i=\frac{1}{\alpha_i}$ in Lemma~\ref{lem: lambda1}.

\end{proof}

\subsection{Exact solutions using Tanh method}\label{subsec: exact solutions using Tanh method}

Enlightened by the works of \cite{hung2012JJIAM,CHMU,Rodrigo&Mimura00,Rodrigo&Mimura01}, our idea is to look for a monotone solution with a hyperbolic tangent profile. We make the following ans\"atz for solving \eqref{eqn:  2 species exact TWS example}:
\begin{equation}\label{eqn: ansatz for solns}
\begin{cases}
\vspace{3mm}
u(x) = k_1\,\big(1-\tanh x\big)^2, \ \ &x\in\mathbb{R},\\
v(x) = k_2\,\big(1+\tanh x\big), \ \ &x\in\mathbb{R},\\
\end{cases}
\end{equation}
where $k_1$ and $k_2$ are positive constants to be determined. Since the derivative of $\tanh x$ is expressible in terms of itself, i.e. $\displaystyle\frac{d}{dx}\tanh x=1-\tanh^2 x$, we see that the $n$th derivative of a polynomial in $\tanh x$ with any order is also a a polynomial in $\tanh x$. Inserting  ans\"{a}tz \eqref{eqn: ansatz for solns} into \eqref{eqn:  2 species exact TWS example}, this fact enables us to get
\begin{eqnarray}\nonumber
d_1(u^2)_{xx}+\theta\,u_x+u\,(\sigma_1-c_{11}\,u-c_{12}\,v) & = & 
u\,\big(\zeta_0+\zeta_1\,T(x)+\zeta_2\,T^2(x)+\zeta_3\,T^3(x)\big), \\
d_2(v^2)_{xx}+\theta\,v_x+v\,(\sigma_2-c_{21}\,u-c_{22}\,v) & = & 
v\,\big(\xi_0+\xi_1\,T(x)+\xi_2\,T^2(x)+\xi_3\,T^3(x)\big),\notag
\end{eqnarray}
where $T(x):=\tanh x$,
\begin{eqnarray}
\label{}
  \zeta_0&=&-c_{11} \,k_1^2-c_{12}\, k_1\,k_2+12\,d_1\,k_1-2\,\theta +\sigma_1\,k_1,   \\
  \zeta_1&=&4\,c_{11}\,k_1^2+c_{12}\,k_1\,k_2+8\,d_1\,k_1-2\,\theta-2\,\sigma_1\,k_1,   \\
  \zeta_2&=&-6\,c_{11}\,k_1^2+c_{12}\,k_1\,k_2-32\,d_1\,k_1+\sigma_1\,k_1,  \\
  \zeta_3&=&4\,c_{11}\,k_1^2-c_{12}\,k_1\,k_2-8\,d_1\,k_1,   \\
  \zeta_4&=&20\,d_1\,k_1-c_{11}\,k_1^2,    
\end{eqnarray}
and
\begin{eqnarray}
\label{}
  \xi_0&=&-c_{22}\,k_2^2-c_{21}\,k_1\,k_2+2\,d_2\,k_2+\theta+\sigma_2\,k_2,   \\
  \xi_1&=&-2\,c_{22}\,k_2^2+c_{21}\,k_1\,k_2-6\,d_2\,k_2-\theta+\sigma_2\,k_2,   \\
  \xi_2&=&-c_{22}\,k_2^2+c_{21}\,k_1\,k_2-2\,d_2\,k_2,  \\
  \xi_3&=&6\,d_2\,k_2-c_{21}\,k_1\,k_2.   
\end{eqnarray}
Equating the coefficients of powers of $T(x)$ to zero yields a system of 9 equations:
\begin{alignat}{3}\label{eqn: algebraic eqn tanh method}
\zeta_i & =0 \ \ (i=0,1,2,3,4), & \qquad 
\xi_i & =0 \ \ (i=0,1,2,3).
\end{alignat}
It turns out that, with $d_1$, $d_2$, $c_{11}$, and $c_{22}$ being free parameters, \eqref{eqn: algebraic eqn tanh method} can be solved to give
\begin{alignat}{4}\label{eqn: soln of algebraic eqn tanh method}
k_1 & =\frac{20\,d_1}{c_{11}}, & \qquad 
\sigma_1 & =80\,d_1, & \qquad 
c_{12} & =\frac{18\,c_{22}\,d_1}{d_2}, & \qquad
\theta & =0
,\\
k_2 & =\frac{4\,d_2}{c_{22}}, & \qquad 
\sigma_2 & =8\,d_2, & \qquad 
c_{21} & =\frac{3\,c_{11}\,d_2}{10\,d_1}. & \qquad
\notag
\end{alignat}
 The result obtained is summarized in the following
\begin{theorem}\label{thm: exact TWS appendix}
System \eqref{eqn:  2 species exact TWS example} has a solution of the
form \eqref{eqn: ansatz for solns} provided that \eqref{eqn: soln of algebraic eqn tanh method} holds.

\end{theorem}

\subsection{Exact solutions of \textbf{(SKT-tw)}}\label{subsec: exact solutions of SKT}


Inspired by the exact periodic solutions proposed in \cite{Guedda-Kersner-Klincsik-Logak-14-Exact-fronts-periodic}, we make the ans\"atz for solving \eqref{eqn: 3 species TWS nonexistence} as follows:
\begin{equation}\label{eqn: ansatz for solns Cos}
\begin{cases}
\vspace{3mm}
u(x) = k_1+m_1\,\cos\,(\mu\,x), \ \ &x\in\mathbb{R},\\
\vspace{3mm}
v(x) = k_2+m_2\,\cos\,(\mu\,x), \ \ &x\in\mathbb{R},\\
w(x) = k_3+m_3\,\cos\,(\mu\,x), \ \ &x\in\mathbb{R},
\end{cases}
\end{equation}
where $\mu\neq0$, $k_1$, $k_2$, $k_3>0$ and $m_1\neq0$, $m_2\neq0$, $m_3\neq0$ with $\left| {m_1} \right|\le k_1$, $\left| {m_2} \right|\le k_2$, and $\left| {m_3} \right|\le k_3$ are constants to be determined.  
Inserting ans\"{a}tz \eqref{eqn: ansatz for solns Cos} into \eqref{eqn: 3 species TWS nonexistence}, we obtain\begin{eqnarray}\nonumber
d_1(u^2)_{xx}+\theta\,u_x+u\,(\sigma_1-c_{11}\,u-c_{12}\,v-c_{13}\,w) &=& 
\zeta_0+\zeta_1\,\mathcal{C}(x)+\zeta_2\,\mathcal{C}^2(x)+\zeta_3\,\mathcal{}S(x), \\ \notag
d_2(v^2)_{xx}+\theta\,v_x+v\,(\sigma_2-c_{21}\,u-c_{22}\,v-c_{23}\,w) &=& 
\xi_0+\zeta_1\,\mathcal{C}(x)+\xi_2\,\mathcal{C}^2(x)+\xi_3\,\mathcal{S}(x), \\ \notag
d_3(w^2)_{xx}+\theta\,w_x+w\,(\sigma_3-c_{31}\,u-c_{32}\,v-c_{33}\,w) &=& 
\varsigma_0+\varsigma_1\,\mathcal{C}(x)+\varsigma_2\,\mathcal{C}^2(x)+\varsigma_3\,\mathcal{S}(x), \notag
\end{eqnarray}
where $\mathcal{C}(x):=\cos\,(\mu\,x)$, $\mathcal{S}(x):=\sin\,(\mu\,x)$ and 
\begin{eqnarray}
\label{}
  \zeta_0&=&-c_{11}\,k_1^2-c_{12}\,k_2\,k_1-c_{13}\,k_3\,k_1+2\,d_1\,\mu ^2\,m_1^2+k_1\,\sigma _1,   \\
  \zeta_1&=&-2\,c_{11}\,k_1\,m_1-c_{12}\,k_2\,m_1-c_{13}\,k_3\,m_1-c_{12}\,k_1\,m_2 \\ \notag
               & &-c_{13}\,k_1\,m_3-2\,d_1\, k_1\,\mu ^2\,m_1+m_1\,\sigma _1,   \\ 
  \zeta_2&=&-c_{11}\,m_1^2-c_{12}\,m_2\,m_1-c_{13}\,m_3\,m_1-4\,d_1\,\mu ^2\,m_1^2,  \\
  \zeta_3&=&\theta\,\mu\,m_1, 
\end{eqnarray}
\begin{eqnarray}
\label{}
  \xi_0&=&-c_{22}\,k_2^2-c_{21}\,k_1\,k_2-c_{23}\,k_3\,k_2+2\,d_2\,\mu^2\,m_2^2+k_2\,\sigma _2,   \\
  \xi_1&=&-c_{21}\,k_2\,m_1-c_{21}\,k_1\,m_2-2\,c_{22}\,k_2\,m_2-c_{23}\,k_3\,m_2 \\ \notag
               & &-c_{23}\,k_2\,m_3-2\,d_2\,k_2\,\mu ^2\,m_2+m_2\,\sigma _2,   \\ 
  \xi_2&=&-c_{22}\,m_2^2-c_{21}\,m_1\,m_2-c_{23}\,m_3\,m_2-4\,d_2\,\mu ^2\,m_2^2,  \\
  \xi_3&=&\theta\,\mu\,m_2, 
\end{eqnarray}
\begin{eqnarray}
\label{}
  \varsigma_0&=&-c_{33}\,k_3^2-c_{31}\,k_1\,k_3-c_{32}\,k_2\,k_3+2\,d_3\,\mu ^2\,m_3^2+k_3\,\sigma _3,   \\
  \varsigma_1&=&-c_{31}\,k_3\,m_1-c_{32}\,k_3\,m_2-c_{31}\,k_1\,m_3-c_{32}\,k_2\,m_3 \\ \notag
               & &-2\,c_{33}\,k_3\,m_3-2\,d_3\,k_3\,\mu ^2\,m_3+m_3\,\sigma _3,   \\ 
  \varsigma_2&=&-c_{33}\,m_3^2-c_{31}\,m_1\,m_3-c_{32}\,m_2\,m_3-4\,d_3\,\mu ^2\,m_3^2,  \\
  \varsigma_3&=&\theta\,\mu\,m_3.
\end{eqnarray}
Equating the coefficients of powers of $\mathcal{C}(x)$ and $\mathcal{S}(x)$ to zero yields a system of 12 equations:
\begin{alignat}{3}\label{eqn: Cos method}
\zeta_i & =0 \ \ (i=0,1,2,3), & \qquad 
\xi_i & =0 \ \ (i=0,1,2,3), & \qquad 
\varsigma_i & =0 \ \ (i=0,1,2,3).
\end{alignat}
It turns out that, with $m_i$, $d_i$, $c_{ij}$ $(i,j=1,2,3,i\neq j)$, and $\mu$ being free parameters, \eqref{eqn: Cos method} can be solved to give
\begin{alignat}{4}\label{eqn: soln of algebraic eqn tanh method}
k_1 & =-m_1, & \qquad 
\sigma_1 & =2\,\left(c_{12}\,m_2+c_{13}\,m_3+3\,d_1\,\mu ^2\,m_1\right), & \qquad 
c_{11} & =-m_1^{-1}(c_{12}\,m_2+c_{13}\,m_3+4\,d_1\,\mu ^2\,m_1), & \qquad
\\ \notag
k_2 & =m_2, & \qquad 
\sigma_2 & =-2\,\left(c_{21}\,m_1+3\,d_2\,\mu ^2\,m_2\right), & \qquad 
c_{22} & =-m_2^{-1}(c_{21}\,m_1+c_{23}\,m_3+4\,d_2\,\mu ^2\,m_2), & \qquad
\\ \notag
k_3 & =m_3, & \qquad 
\sigma_3 & =-2\,\left(c_{31}\,m_1+3\,d_3\,\mu ^2\,m_3\right), & \qquad 
c_{33} & =-m_3^{-1}(c_{31}\,m_1+c_{32}\,m_2+4\,d_3\,\mu ^2\,m_3), & \qquad
\\ \notag
\theta & =0.
\end{alignat}
We note that $\zeta_3=\xi_3=\varsigma_3=0$ immediately leads to $\theta=0$. The result obtained is summarized in the following

\begin{theorem}\label{thm: exact TWS appendix Cosine soln}
System \eqref{eqn: 3 species TWS nonexistence} has a solution of the form \eqref{eqn: ansatz for solns Cos} provided that \eqref{eqn: soln of algebraic eqn tanh method} holds.
\end{theorem}

In view of Theorem~\ref{thm: exact TWS appendix Cosine soln}, \eqref{eqn: 3 species TWS nonexistence} has the solution
\begin{equation}\label{eqn: ansatz for solns Cos example}
\begin{cases}
\vspace{3mm}
u(x) = \displaystyle\frac{1}{10}\big(1-\cos\,(2\,x)\big), \ \ &x\in\mathbb{R},\\
\vspace{3mm}
v(x) = \displaystyle\frac{1}{11}\big(1+\cos\,(2\,x)\big), \ \ &x\in\mathbb{R},\\
w(x) = \displaystyle\frac{1}{12}\big(1+\cos\,(2\,x)\big), \ \ &x\in\mathbb{R},
\end{cases}
\end{equation}
when $d_i=\sigma_i=c_{ii}=1$ $(i=1,2,3)$, $c_{12}= \displaystyle\frac{1067}{60}$, $c_{13}=1$, $c_{21}= \displaystyle\frac{175}{11}$, $c_{23}= \displaystyle\frac{6}{11}$, $c_{31}=15$, $c_{32}= \displaystyle\frac{11}{12}$, and $\theta=0$. The resulting profiles of \eqref{eqn: ansatz for solns Cos example} are shown in Figure~\ref{fig: Morisita-ShigesadaPeriodicSoln}.

\begin{figure}[ht!]
\centering
\mbox{
{\includegraphics[width=0.60\textwidth]{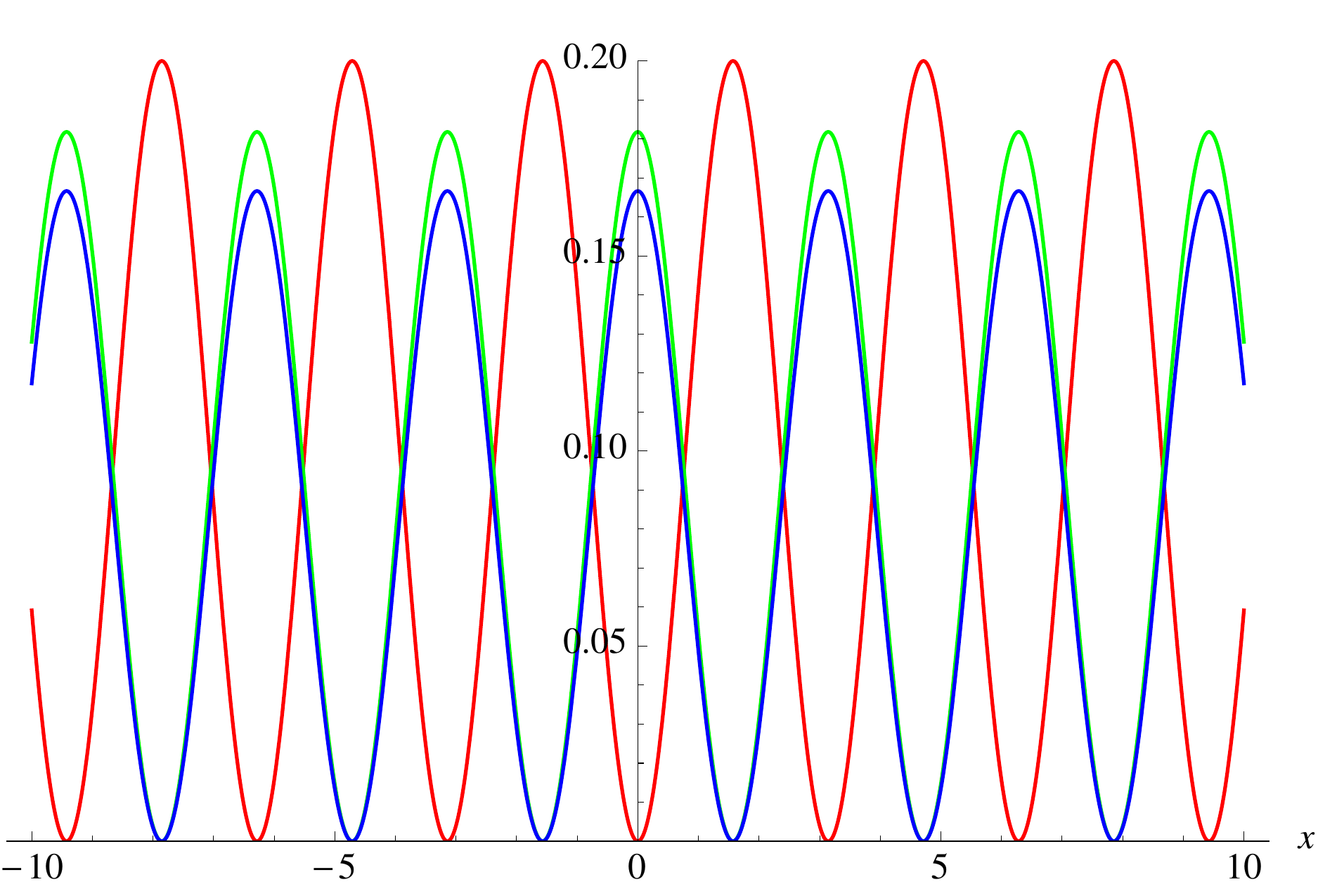}
             } \quad \hspace{0mm}
             }
\caption{\small Red: $u(x) =\displaystyle\frac{1}{10}\big(1-\cos\,(2\,x)\big)$; green: $v(x) =\displaystyle\frac{1}{11}\big(1+\cos\,(2\,x)\big)$; blue: $w(x) =\displaystyle\frac{1}{12}\big(1+\cos\,(2\,x)\big)$. 
\label{fig: Morisita-ShigesadaPeriodicSoln}}
\end{figure}

\end{document}